\documentclass[a4paper,11pt]{article}

\usepackage[toc,page]{appendix}

\usepackage{amsthm}

\usepackage{amsmath}
\usepackage{latexsym}
\usepackage{amssymb}
\usepackage{verbatim}
\usepackage{setspace}
\usepackage{qtree}
\usepackage{enumitem}
\usepackage{xcolor}
\usepackage{url}
\usepackage{hyphenat}
\usepackage{hyperref}

\usepackage{authblk}

\usepackage[margin=1.4in]{geometry}

\usepackage{tikz}
\usetikzlibrary{arrows,decorations.pathmorphing,backgrounds,positioning,fit}

\DeclareMathOperator{\Free}{Free}

\author[a]{Reijo Jaakkola}
\author[b]{Antti Kuusisto}
\affil[a,b ]{Tampere University and University of Helsinki, Finland}

\date{}

\begin{document}

\setlength\abovedisplayskip{3pt}
\setlength\belowdisplayskip{3pt}
\title{Algebraic classifications for fragments of
first-order logic and beyond}

\theoremstyle{plain}
\newtheorem{theorem}{Theorem}[section]
\newtheorem{lemma}[theorem]{Lemma}
\newtheorem{corollary}[theorem]{Corollary}
\newtheorem{proposition}[theorem]{Proposition}
\theoremstyle{definition}
\newtheorem{definition}[theorem]{Definition}
\newtheorem{remark}[theorem]{Remark}
\newtheorem{example}[theorem]{Example}

\maketitle

\begin{abstract}
\noindent
Complexity and decidability of logics is a major research area involving a huge range of different logical systems. This calls for a unified and systematic approach for the field. We introduce a research program based on an algebraic approach to complexity classifications of fragments of first-order logic (FO) and beyond. Our base system GRA, or general relation algebra, is equiexpressive with FO. It resembles cylindric algebra but employs a finite signature with only seven different operators. We provide a comprehensive classification of the decidability and complexity of the systems obtained by limiting the allowed sets of operators. We also give algebraic characterizations of the best known decidable fragments of FO. Furthermore, to move beyond FO, we introduce the notion of a generalized operator and briefly study related systems.
\end{abstract}


\section{Introduction}

The failure of Hilbert's program and the realization of the undecidability of
first-order logic $\mathrm{FO}$ put an end to the most prestigious 
plans of automating mathematical reasoning. However, research with more
modest aims continued right away. Perhaps the most direct descendant of Hilbert's 
program was the work on the \emph{classical decision problem}, i.e., the initiative to
classify the quantifier prefix classes of $\mathrm{FO}$ according to 
whether they are decidable or not. This major program was successfully 
completed in the 1980's, see \cite{borger97} for an overview.

Subsequent work has been more scattered but highly active.
Currently, the state of the art on 
decidability and complexity of fragments of $\mathrm{FO}$
divides roughly into two branches:
research on variants of \emph{two-variable logic} $\mathrm{FO}^2$
and the \emph{guarded fragment} $\mathrm{GF}$. 
Two-variable logic $\mathrm{FO}^2$ is the fragment of $\mathrm{FO}$ where
only two variable symbols $x,y$ are allowed.
It was proved decidable in \cite{mortimer}
and \textsc{NexpTime}-complete in \cite{GKV97}.
The extension of $\mathrm{FO}^2$ with counting quantifiers, known as $\mathrm{C}^2$,
was proved decidable in \cite{GOR97,PST97}
and \textsc{NexpTime}-complete in \cite{PH05}.
Research on variants of $\mathrm{FO}^2$ is currently very active.
Recent work has focused on complexity issues
in restriction to particular structure classes and also questions
related to built-in relations, see, e.g., 
\cite{Be16, Witek16, Ki14, kihate, kopcz, Zeume13, torun}
for a selection of recent contributions.
See also \cite{U1a, U1b} where the \emph{uniform one-dimensional
fragment} $\mathrm{U}_1$ is defined. This system  
extends $\mathrm{FO}^2$ to a logic that allows an arbitrary 
number of variables but preserves most of the relevant metalogical
properties, including the \textsc{NexpTime}-completeness of the
satisfiability problem. The article \cite{u1surv} provides a 
survey on $\mathrm{U_1}$.

The guarded fragment $\mathrm{GF}$ was initially conceived as an
extension of modal logic, being a system where quantification is
similarly localized
as in the Kripke semantics for modal logic. After its introduction in
\cite{andreka}, it was soon proved \textsc{2ExpTime}-complete in 
\cite{gradel99}. The guarded fragment has proved successful in 
relation to applications, and it has been extended in
several ways. The \emph{loosely guarded} \cite{loosely},
\emph{clique guarded} \cite{clique} and
\emph{packed} \cite{packed} fragments
impose somewhat more liberal
conditions than $\mathrm{GF}$ for keeping quantification
localized, but the basic idea is the same. All these 
logics have the same
\textsc{2ExpTime}-complete complexity as $\mathrm{GF}$
(see, e.g., \cite{barany}). The more
recently introduced \emph{guarded
negation fragment} $\mathrm{GNFO}$ 
\cite{barany} is a very
expressive extension of $\mathrm{GF}$ based on 
restricting the use of negation in the same way $\mathrm{GF}$
restricts quantification. The logic $\mathrm{GNFO}$
also extends the \emph{unary negation fragment} $\mathrm{UNFO}$
\cite{segoufin} which is
orthogonal to $\mathrm{GF}$ in expressive power.
Despite indeed being quite expressive, $\mathrm{GNFO}$ shares
the \textsc{2ExpTime}-completeness
of $\mathrm{GF}$, and so does $\mathrm{UNFO}$.

Complexity of fragments of $\mathrm{FO}$ is
important also in 
knowledge representation, especially in relation to 
\emph{description logics} \cite{descriptionlogichandbook}.
In this field, complexities
are classified in great detail, 
operator by operator. The \emph{Description Logic
Complexity Navigator} at \url{http://www.cs.man.ac.uk/~ezolin/dl/}
provides an extensive and detailed taxonomy of the
best known relatively recent results.
Most description logics limit to vocabularies with
at most binary relations, but there are notable exceptions, e.g.,
the $\textsc{ExpTime}$-complete
logic $\mathcal{DLR}$ \cite{Calvanese}.

Somewhat less extensively studied decidable fragments of $\mathrm{FO}$
include the \emph{Maslov class} \cite{Maslov}; 
the \emph{fluted logic} \cite{quinefluted},
\cite{flutedlidiatendera}; the 
\emph{binding form} systems \cite{Mogavero}
and generalizations of prefix classes in, e.g.,
\cite{Voigt17}. We should also mention the \emph{monadic fragment} of $\mathrm{FO}$---possibly the first
non-trivial $\mathrm{FO}$-fragment shown
decidable \cite{Loovenheimijoojoo}. Of these frameworks,
research relating to \cite{Voigt17} and work on 
fluted logic has recently been active. The systems 
of \cite{Voigt17}
are largely based on limiting how the variables of
different atoms can overlap, 
while fluted logic restricts how variables
can be permuted.

While the completion of the classical decision problem in 
the 1980s was a major achievement, that project concentrated only on a
very limited picture of $\mathrm{FO}$: prefix classes only.
The restriction to prefix classes can be seen as a \emph{strong limitation}, both from the
theoretical and applied perspectives.
The subsequent research trends---e.g., the
work on 
the guarded fragment, $\mathrm{FO}^2$ 
and description logics---of course
lifted this limitation, leading to a more liberal theory with a 
wide range of applications from verification to database theory and
knowledge representation. Indeed, as we have seen above,
the current state of the art
studies a \emph{huge number} of different logical frameworks, tailored
for different purposes. However, consequently
the \emph{related research is 
scattered}, and could surely
\emph{benefit from a more systematic~approach}.
The current article suggests a framework for such a systematic approach.

\medskip

\noindent
\textbf{Our contributions.}
We introduce a research program for 
classifying complexity and decidability of
fragments of $\mathrm{FO}$ (and beyond) within an
\emph{algebraic framework}.
To this end, we define an algebraic system
designed to enable a \emph{systematic and fine-grained approach} to
classifying first-order fragments.
One of the key ideas is to identify a \emph{finite}
collection of operators to capture 
the expressive power of $\mathrm{FO}$, so our 
algebra has a finite signature. 
In $\mathrm{FO}$, there are essentially infinitely many
quantifiers $\exists x_i$ due to the different variable
symbols $x_i$, and this issue gives rise to
the infinite signature of \emph{cylindric set algebras}, which are 
the principal algebraic formulation of $\mathrm{FO}$. Basing our
investigations on finite signatures
leads to a highly controlled setting
that directly elucidates how the
expressive power of $\mathrm{FO}$ arises. This is achieved by listing a
finite set of operators that the expressivity of $\mathrm{FO}$ is based on.

The principal system we introduce is built on the
algebraic signature $$(e,p,s,I,\neg, J,\exists).$$ The atomic 
terms of the related algebra are simply
relation symbols (of any arity), and 
complex algebraic terms are built from atoms by applying the
operators in the signature
in the usual way.
This defines an 
algebra over every relational structure $\mathfrak{M}$. The
atomic terms $R$ are interpreted as the
corresponding
relations $R^{\mathfrak{M}}$.
The operators correspond to functions that modify 
relations into new relations over $\mathfrak{M}$ as follows. 
\begin{enumerate}

\item
$\neg$ is the complementation operator.
\item
$J$ is the join operator.
\item
$\exists$ is the existential 
quantification (or projection)
operator.
\item
$p$ is a cyclic permutation operator.
\item
$s$ a swap operator (swapping the last two elements of 
tuples).
\item
$I$ is an identification (or substitution) operator.
This operator deletes tuples whose
last two members are not identical and then projects away the last
coordinate of the remaining tuples. The operator allows us to 
perform operations that in standard (non-algebraic) FO would correspond to 
variable substitutions.
\item
$e$ is the constant operator
denoting the equality (or identity) relation over
the domain of the model $\mathfrak{M}$.
\end{enumerate}

We let $\mathrm{GRA}(e,p,s,I,\neg, J,\exists)$ refer to
the system based on these operators, 
with GRA standing for \emph{general relation algebra}.
To simplify notation, we also let $\mathrm{GRA}$
stand for $\mathrm{GRA}(e,p,s,I,\neg, J,\exists)$ in
the current article.
Furthermore, by $\mathrm{GRA}\setminus f_1,\dots , f_k$ we
refer to $\mathrm{GRA}$
with the operators $f_1,\dots, f_k
\in\{e,p,s,I,\neg, J,\exists\}$ removed.

We begin our study by proving that
\emph{$\mathrm{GRA}$ and $\mathrm{FO}$ are 
equiexpressive.} The next aim is to classify the decidability and 
complexity properties of the principal
subsystems of $\mathrm{GRA}$. Firstly, 
$\mathrm{GRA}\setminus \neg$ is trivially decidable, every term
being satisfiable.
Nevertheless, $\mathrm{GRA}\setminus \neg$ is interesting as we show it
can define precisely all conjunctive queries with equality.
Then we establish
that $\mathrm{GRA}\setminus \exists$ is \textsc{NP}-complete.
We then show that satisfiability of $\mathrm{GRA}\setminus J$
can be checked by a finite automaton, and futhermore, we 
prove $\mathrm{GRA}\setminus I$ to be \textsc{NP}-complete. 
We thereby identify new decidable low-complexity 
fragments of FO. Including the discovery of the algebraic systems, we 
identify, e.g., the $\mathrm{NP}$-complete fragment $\mathcal{F}$ of $\mathrm{FO}$
based on the restriction that when forming conjunctions 
$\varphi(x_1,\dots , x_m) \wedge \psi(y_1,\dots , y_n)$,
the sets $\{x_1,\dots , x_m\}$
and $\{y_1,\dots , y_n\}$ of variables should be disjoint.

On the negative side, we 
show that $\mathrm{GRA}(p,I,\neg,J,\exists)$ is $\Pi_1^0$-complete, so
removing either $e$ or $s$ (or both)
from $\mathrm{GRA}$ does \emph{not} lead to 
decidability. 
Thus we have the following close to complete
first classification: removing
any of the operators $\neg,\exists,I,J$ gives a
decidable system, while dropping $e$ or $s$ (or both)
keeps the system undecidable. The only open case
concerning the removal of a single operator is $\mathrm{GRA} \setminus p$.
We leave the study of the complexity and decidability of
subsystems of $\mathrm{GRA}$ there in this introductory article.

To push our program further, we define a general notion of a
\emph{relation operator} which essentially puts connectives and
(generalized) quantifiers under the same umbrella notion.
The definition can be seen as a slight generalization of
the notion of a generalized quantifier due to
Mostowski \cite{mostowski} and Lindstr\"{o}m \cite{lindstrom}. 
We then study 
variants of $\mathrm{GRA}$ with different sets of 
relation operators.

 In particular, we characterize
the guarded fragment, two-variable logic and 
fluted logic by different algebras. The
guarded fragment corresponds 
to the algebra $\mathrm{GRA}(e,p,s,\setminus,\Dot\cap,\exists)$
where the symbol $\setminus$ denotes the relative 
complementation operator and $\Dot{\cap}$ is a
\emph{suffix intersection} operator. The suffix intersection is an operator that is similar to 
standard intersection but makes sense also when intersecting relations of
different arities.
Two-variable
logic---over vocabularies with at most binary
relation symbols---corresponds to $\mathrm{GRA}(e,s,\neg,\Dot{\cap},\exists)$,
and fluted logic
turns out to be $\mathrm{GRA}(\neg,\Dot{\cap},\exists)$.
%
%
%

%
%
%
Note that the algebras for fluted logic and two-variable logic 
are clearly rather intimately related (note that we do not impose 
restrictions on relation symbol arities for fluted logic). 
Also, since the guarded fragment is characterized by $\mathrm{GRA}(e,p,s,\setminus,\Dot\cap,\exists)$,
and since $\mathrm{GRA}(e,s,\neg,\Dot\cap,\exists)$ 
and $\mathrm{GRA}(e,p,s,\neg,\Dot\cap,\exists)$ can be shown
equiexpressive over vocabularies with at most binary relations, we
observe that also two-variable logic and the
guarded fragment are very nicely and closely linked.
These kinds of results demonstrate the
explanatory power and potential usefulness of the
idea of comparing $\mathrm{FO}$-fragments \emph{under the
same umbrella framework} based on 
different kinds of \emph{finite signature algebras}. Indeed, \emph{each finite 
operator set specifies what the building blocks of the related 
logic are---precisely.}
In this article we give characterizations mainly to some of the main decidable 
fragments of FO, but a natural future research
direction involves also pushing these studies
beyond first-order logic by using the notion of
relation operator defined here. Also, finding different operator
sets that are expressively complete for FO is a relevant question.



The contributions of this article can be 
summarized as follows. \smallskip
\begin{enumerate}
    \item
    The main objective is to \emph{introduce the program} of systematically
    classifying the complexity and decidability properties of logics
    with \emph{different finite signature algebras}. We believe it is 
    useful to introduce the idea, as no previous analogous
    study on first-order fragments exists.
    %
    %
    %
    \item
    Concretely, we provide a comprehensive
    classification of the principal subsystems of 
    $\mathrm{GRA}(e,p,s,I,\neg, J, \exists)$ (which itself 
    characterizes 
    $\mathrm{FO}$).
    %
    In each solved case, we also pinpoint the complexity of the system.
    In the process, we identify interesting new low-complexity fragments of FO. 
    \item
    We find algebraic characterizations for $\mathrm{FO}$ and 
    some of its main decidable fragments such as
    $\mathrm{FO}^2$,
    the guarded fragment and fluted logic. (Furthermore, we additionally find algebras that
    characterize conjunctive queries, equality-free FO, quantifier-free FO
    and the set of first-order atoms.) We also provide a
    \textsc{2ExpTime}-completeness result for the algebra of for the guarded fragment GF.
    This turns out to require quite intricate proof techniques and new
    notions (e.g., the notion of a \emph{term guard}) to keep
    the translations between GF and the algebra polynomial. 
    Showing \textsc{NexpTime}-completeness of the
    algebra for $\mathrm{FO}^2$ turns out easier.
    \item
    We define a general notion of a relation operator.
    This relates also to
    further directions for our 
    future research summarized in the concluding section.
\end{enumerate}

As yet further work, we also
consider $\mathrm{GRA}(e,s,\backslash,\Dot{\cap},\exists)$
and show it \textsc{ExpTime}-complete. This algebra can be
seen as a new, decidable fragment of $\mathrm{FO}$. Furthermore, we
characterize the unary negation fragment by using a 
unary negation operator as well as suitable join operators.

\medskip

\noindent
\textbf{Further notes on related work.}
We already extensively surveyed the \emph{related work} 
concerning our program. We now
give further related information on
algebraic issues.

There are various algebraic approaches to $\mathrm{FO}$,
e.g., Tarski's \emph{cylindric algebras},
their semantic counterparts \emph{cylindric set algebras}
and the \emph{polyadic algebras} of Halmos.
The book \cite{hirsch} gives a
comprehensive and relatively recent
account of the these systems.
Also, variants of Codd's \emph{relational algebra} \cite{codd} are
important, although the main systems studied within
the related database-theory-oriented setting are
not equivalent to standard $\mathrm{FO}$ but instead
relate to domain independent first-order logic. 
The closest approach to our system is Quine's
\emph{predicate functor logic}
\cite{quine36}, \cite{quine60}, \cite{quine72}.
This system comes in several variants, with different sets of
operators used. But the spirit of the approach bears some similarity to
the main algebraic system we study in this article.
Variants of predicate functor logic can be naturally
considered to be within the scope of
our research program. Predicate functor logic has been studied 
very little, and we are not aware of
any work relating to complexity theory that 
has appeared before our current work.
The notable works within 
this framework include the complete axiomatizations given in
\cite{kuhn} and \cite{bacon}. Concerning further algebraic settings, Tarski's \emph{relation 
algebra} (see \cite{hirsch}) is also related to our work, but focuses on
binary relations.

This paper is an extended, full version of the conference
paper \cite{csl2023version}. 
The article \cite{newversion} is an early
version of the current article. The paper \cite{newversion} has already been followed-up by  \cite{ORDEREDFRAGMENTS}, where several natural extensions of so-called \emph{ordered logic}, fluted logic and $\mathrm{FO}^2$ were studied within the algebraic framework introduced below.  
The paper \cite{oldversion} is the first
version of \cite{newversion}, containing many of the results below, but using a slightly different set of
algebraic operators. The research program realized in the current article and its earlier versions was proposed in \cite{games19}, which also discusses the $\mathrm{FO}$-equivalence of the operators of \cite{oldversion} and suggests, for example, studying systems with limited permutations.

\section{Preliminaries}

Let $A$ be an arbitrary set.
As usual, a \textbf{$k$-tuple} over $A$ is an
element of $A^k$. When $k=0$, we let $\epsilon$
denote the unique \textbf{$0$-tuple} in $A^k = A^0$.
Note that $A^0 = B^0 = \varnothing^0 = \{\epsilon\}$
for all sets $A$ and $B$.
Note also that $\varnothing^k =\varnothing$ for all positive integers $k$.
If $k$ is a non-negative integer, then a $k$-ary
\textbf{AD-relation}
over a set $A$ is a pair $(R,k)$ where $R\subseteq A^k$ is a $k$-ary
relation in the usual sense, i.e., simply a set of $k$-tuples.
Here `AD' stands for  \emph{arity definite}.
We call $(\varnothing, k)$ the \textbf{empty $k$-ary AD-relation}.
For a non-negative integer $k$, we let 
$\top_k$ (respectively, $\bot_k$) denote an operator 
that maps any set $A$ to the AD-relation $\top_k(A) := (A^k, k)$
(respectively, $\bot_k(A) := (\varnothing,k)$).
We may write $\top_k^A$ for $\top_k(A)$ and simply $\top_0$
for $\top_0(A)$, and we 
may write $\bot_k^A$ or $\bot_k$ for $\bot_k(A)$.
We note that $\bot_k^\varnothing = \top_k^\varnothing$
iff $k \not= 0$. 
When $T = (R,k)$ is a $k$-ary AD-relation, we
let $\mathit{rel}(T)$ denote $R$ and 
write $\mathit{ar}(T) = k$ to refer to the arity of $T$.
If $T$ is an AD-relation, $t\in T$ 
\emph{always means that} $t\in \mathit{rel}(T)$.

The notion of a model is defined as usual in model theory, assuming
model domains are never empty. For simplicity, we restrict attention to
relational models, i.e., vocabularies of models do not contain
function or constant symbols.
We use the convention where the domain of a
model $\mathfrak{A}$ is
denoted by $A$, the domain of $\mathfrak{B}$ by $B$ et cetera.
We let $\hat{\tau}$ denote the \textbf{full relational vocabulary}
containing countably infinitely many relation symbols of
each arity $k\geq 0$.
We let $\mathrm{VAR} = \{v_1,v_2,\dots \}$ denote the countably 
infinite set of exactly all variables used in 
first-order logic $\mathrm{FO}$. We also use metavariables (e.g.,
$x,y,z,x_1,x_2\dots$) to refer to symbols in $\mathrm{VAR}$.
The syntax of $\mathrm{FO}$ is built in the usual way, 
starting from the set of atoms 
consisting of \textbf{equality atoms} (i.e., atoms 
with the equality symbol $=$)
and \textbf{relation atoms} $R(x_1,\dots , x_n)$ where $R\in \hat{\tau}$. When writing relation
atoms, we often write $Rx_1,\dots x_n$
instead of $R(x_1,\dots , x_n)$ to
simplify notation.
By an $\mathrm{FO}$-formula $\varphi(x_1,\dots, x_k)$ we 
refer to a formula whose free
variables are exactly $x_1,\dots, x_k$.
An $\mathrm{FO}$-formula $\varphi$ (without a list of
variables) may or may not have free variables. The set of
free variables of $\varphi$ is denoted by $Free(\varphi)$.

Now, suppose $(x_1,\dots , x_k)$ is a
tuple of pairwise distinct variables and consider a formula $\varphi(x_1,\dots , x_k)$.
Suppose, likewise, that also $(y_1,\dots , y_k)$ is a
tuple of pairwise distinct variables. Then we let $\varphi(y_1,\dots , y_k)$ denote
the formula obtained from $\varphi(x_1,\dots , x_k)$ by simultaneously replacing each free variable $x_i$ by $y_i$
for each $i\leq k$ (and avoiding variable capture in the process by renaming bounded variables suitably,
if necessary).

Let $k\geq 0$ and
consider an $\mathrm{FO}$-formula $\varphi(v_{i_1},\dots ,v_{i_k})$
where $i_1< \dots < i_k$.
The formula $\varphi(v_{i_1},\dots ,v_{i_k})$
\textbf{defines} the AD-relation $$\bigl(\{(a_1,\dots, a_k)\in A^k\ |\
\mathfrak{A}\models\varphi(a_1,\dots , a_k)\, \}, k\bigr)$$ in the model $\mathfrak{A}$. Notice that we make
crucial use of the linear ordering of the
subindices of the variables $v_{i_1},\dots , v_{i_k}$ in $\mathrm{VAR}$.
We let $\varphi^{\mathfrak{A}}$ denote
the AD-relation defined by $\varphi$ in $\mathfrak{A}$.
Notice---to give an example---that $\varphi(v_1,v_2,v_3)$
and $\varphi(v_6,v_8,v_9)$ define the same AD-relation
over any model. It is important to
recall this phenomenon below. 
When using the
six metavariables $x,y,z,u,v,w$, we henceforth
always assume $(x,y,z,u,v,w)  = (v_{i_1},v_{i_2},v_{i_3},
v_{i_4},v_{i_5}, v_{i_6})$
for some indices $i_1<i_2 < i_3 < i_4 < i_5 < i_6$.
Now, to clarify a further
technical issue, let us consider the
formulas $R(v_1,v_2)$ and $R(v_1,v_1)$. 
Observe that while $R(v_1,v_2)$
defines a binary AD-relation,
the second atom $R(v_1,v_1)$
defines a unary AD-relation since
the only free variable in it is $v_1$.
Consider
the formulas $\varphi := v_1\not=v_1$
and $\psi := v_1\not=v_1 \wedge v_2\not = v_2$.
Now $\varphi^{\mathfrak{A}}$ is the
empty unary AD-relation
and $\psi^{\mathfrak{A}}$ the
empty binary AD-relation. The 
negated formulas  $\neg\varphi$ and $\neg \psi$ then define
the universal unary
and binary AD-relations $(\neg\varphi)^{\mathfrak{A}} = (A,1)$
and $(\neg\psi)^{\mathfrak{A}} = (A\times A,2)$, respectively.
These are two different AD-relations, of course.
This demonstrates the intuition behind the choice to
consider AD-relations rather than ordinary
relations: if $\varphi$  and $\psi$ both defined
the ordinary empty relation $\varnothing$ in $\mathfrak{A}$, then the
action of $\neg$ in $\mathfrak{A}$ on the input $\varnothing$ would
appear ambiguous.
Ordinary relations suffice for most purposes of
studying $\mathrm{FO}$, but we need to be more careful in our detailed algebraic study.

A \textbf{conjunctive query} (CQ) is a
formula $\exists x_1\dots \exists x_k\, \psi$
where $\psi$ is a conjunction of relation atoms $R(y_1,\dots,y_n)$. 
For example 
$\exists y\exists z (Rxyz \wedge Syzuv)$ is a CQ with
the free variables $x,u,v$. 
Conjunctive queries are first
class citizens in database theory.
A \textbf{conjunctive query with
equality} (CQE) is like a CQ but also allows 
equality atoms in addition to relation atoms $R(y_1,\dots,y_n)$.
(We note that $x=y$, for example, is a CQE that defines the identity relation
undefinable by CQs.)

\section{An algebra for first-order logic}\label{FO-algebra}

In this section we define an
algebra equiexpressive with $\mathrm{FO}$.
To this end, consider the
algebraic signature
$\bigl(e,p,s,I,\neg,J,\exists\bigr)$ where
$e$ is an algebraically
nullary symbol (i.e., a constant), the symbols $p,s,I,\neg,\exists$
have arity one, and $J$ has arity two.
Let $\tau$ be a vocabulary, i.e., a set of 
relation symbols. The vocabulary $\tau$
defines a set of \textbf{terms} (or \textbf{$\tau$-terms}) 
built by starting from the 
the symbols $e$ and $R\in \tau$ and composing 
terms by using the symbols $p,s,I,\neg,J,\exists$ in the usual way. 
Thereby $e$ and each $R\in\tau$ are terms,
and if $\mathcal{T}$ and $\mathcal{T}'$ are terms,
then so are $p(\mathcal{T})$, $s(\mathcal{T})$, $I(\mathcal{T})$,
$\neg(\mathcal{T})$, $J(\mathcal{T}$,
$\mathcal{T}')$, $\exists(\mathcal{T})$. We often 
leave out brackets when using unary operators
and write, for example, $Ip R$ instead of $I(p(R))$.
Each term $\mathcal{T}$ is associated with an arity $\mathit{ar}(\mathcal{T})$
which, as we will see later on, equals the arity of the AD-relation that $\mathcal{T}$
defines on a model.
We define that $\mathit{ar}(R)$ is the arity of
the relation symbol $R$, and we define
$\mathit{ar}(e) = 2$;
$\mathit{ar}(p\mathcal{T}) = \mathit{ar}(\mathcal{T})$;
$\mathit{ar}(s\mathcal{T}) = \mathit{ar}(\mathcal{T})$;
%
%
$\mathit{ar}(\neg\mathcal{T}) = \mathit{ar}(\mathcal{T})$;
$\mathit{ar}(J(\mathcal{T},\mathcal{T}'))$
$= \mathit{ar}(\mathcal{T}) + \mathit{ar}(\mathcal{T}')$.
Finally, for $I$ and $\exists$, we define
and $\mathit{ar}(I\mathcal{T})
= \mathit{ar}(\exists\mathcal{T}) = \mathit{ar}(\mathcal{T}) - 1$ if
$\mathit{ar}(\mathcal{T})\geq 1$ and
$\mathit{ar}(I\mathcal{T}) = \mathit{ar}(\exists\mathcal{T}) = 0$ when $\mathit{ar}(\mathcal{T}) = 0$.

Given a model $\mathfrak{A}$ of vocabulary $\tau$,
each $\tau$-term $\mathcal{T}$ defines
an AD-relation $\mathcal{T}^{\mathfrak{A}}$ over $A$.
The arity of $\mathcal{T}^{\mathfrak{A}}$  
will indeed be equal to the arity of $\mathcal{T}$.
Consider terms $\mathcal{T}$ and $\mathcal{S}$ and assume we
have defined AD-relations $\mathcal{T}^{\mathfrak{A}}$
and $\mathcal{S}^{\mathfrak{A}}$.
Then the
below conditions hold.

\smallskip

\medskip

\begin{enumerate}
\item[$R$\, )] Let $R$ be a $k$-ary relation symbol in $\tau$, so $R$ is a
constant term in the algebra. We define

\medskip

$R^{\mathfrak{A}}\, = \, \bigl(\{ (a_1,\dots, a_k)\, |\, \mathfrak{A}\models R(a_1,\dots, a_k)\, \},\,
k\bigr).$

\medskip

\item[$e$\, )] We define $e^{\mathfrak{A}} =
\bigl(\{(a,a)|\, a\in A\},\,2\bigr).$ The constant $e$ is
called the \textbf{equality} constant.
\smallskip
\item[$p$\, )]
If $\mathit{ar}(\mathcal{T}) = k \geq 2$, we define

\medskip

\noindent
\scalebox{0.98}[1]{$(p(\mathcal{T}))^{\mathfrak{A}}
= \bigl(\{(a_{k} , a_1 , \dots , a_{k-1})\, |\,
(a_1,\dots , a_k)\in\mathcal{T}^\mathfrak{A}\, \}, k\bigr),$}

\medskip

where $(a_k, a_1 , \dots , a_{k-1})$ is the $k$-tuple 
obtained from the $k$-tuple
$(a_1,\dots , a_k)$ by 
moving the last element $a_k$ to
the beginning of the tuple.
If $\mathit{ar}(\mathcal{T})$ is $1$ or $0$, we
define $(p(\mathcal{T}))^\mathfrak{A}
= \mathcal{T}^{\mathfrak{A}}$. We call $p$ the \textbf{permutation} 
operator, or \textbf{cyclic permutation} operator.\smallskip
\item[$s$\, )]
If $\mathit{ar}(\mathcal{T}) = k \geq 2$, we define

\medskip

\noindent
\scalebox{0.89}[1]{$\bigl(\mathit{s}(\mathcal{T}))^{\mathfrak{A}}
= \bigl(\{(a_1, \dots , a_{k-2},a_{k},a_{k-1})\, |\, 
(a_1,\dots , a_k)\in\mathcal{T}^\mathfrak{A}\, \}, k \bigr),$}

\medskip

where $(a_1, \dots , a_{k-2},a_{k},a_{k-1})$ is the $k$-tuple
that is obtained from the $k$-tuple $(a_1,\dots , a_k)$ by
swapping the
last two elements $a_{k-1}$ and $a_{k}$ but keeping the 
other elements as they are.
If $\mathit{ar}(\mathcal{T})$ is $1$ or $0$, we
define $(s(\mathcal{T}))^\mathfrak{A}
= \mathcal{T}^{\mathfrak{A}}$.
We refer to $s$ as the \textbf{swap} operator.\smallskip
\item[$I$\, )]
If $\mathit{ar}(\mathcal{T}) = k\geq 2$,
we let $(\mathit{I}(\mathcal{T}))^{\mathfrak{A}}$ be
the AD-relation
%
%
%
%

\medskip

\smallskip

$\bigl(\{(a_1, \dots , a_{k-1})
\, |\, (a_1, \dots , a_{k-1}, a_k)\in\mathcal{T}^\mathfrak{A}\\
\text{ }\hspace{4.3cm}\text{ and }a_{k-1}=a_{k}\}, k-1 \bigr).$

\medskip

\smallskip

If $\mathit{ar}(\mathcal{T})$ is $1$ or $0$, we define $(\mathit{I}(\mathcal{T}))^\mathfrak{A}
= \mathcal{T}^{\mathfrak{A}}$. We
refer to $I$ as the \textbf{identification} operator, or \textbf{substitution} operator. Intuitively, it discards away all
tuples except for those where the two last elements are identical, and then it 
projects away the last one of the last two identical elements of the remaining tuples.\smallskip
\item[$\neg$\, )]
Let $\mathit{ar}(\mathcal{T}) = k$.
We define

\medskip

\scalebox{0.90}[1]{$(\neg(\mathcal{T}))^{\mathfrak{A}}
= \bigl(\{(a_1, \dots , a_{k})\, |\, (a_1,\dots , a_k)
\in A^k\setminus \mathit{rel}(
\mathcal{T}^\mathfrak{A})\, \},\, k\bigr).$}

\medskip

Note in particular
that if $\mathcal{T}^{\mathfrak{A}} = (\varnothing,0) = \bot_0^A$, 
then $( \neg(\mathcal{T}) )^{\mathfrak{A}}
= (\{\epsilon\},0) = \top_0^A$, and
vice versa, if 
$\mathcal{T}^{\mathfrak{A}} =  \top_0^A$, 
then $( \neg(\mathcal{T}) )^{\mathfrak{A}} = \bot_0^A$.
%
%
%
We call $\neg$ the
\textbf{complementation} operator.\smallskip
\item[$J$\, )]
Let $\mathit{ar}(\mathcal{T}) = k$
and $\mathit{ar}(\mathcal{S}) = \ell$. We define $(J(\mathcal{T},\mathcal{S}))^{\mathfrak{A}}$ to be
the AD-relation

\medskip

%
$
\bigl( \{(a_1,\dots , a_k, b_1 , \dots , b_{\ell})\, |\, (a_1,\dots , a_k)
\in \mathcal{T}^\mathfrak{A}\\
\text{ }\hspace{3.4cm}\text{ and }\, (b_1,\dots , b_{\ell})
\in \mathcal{S}^\mathfrak{A}\},k+\ell\bigr).
$
%
%

\medskip

Here we note that $\epsilon$ is interpreted as
the identity of
concatenation, so if $\mathit{rel}(\mathcal{T}^\mathfrak{A}) = \{\epsilon\}$,
then $(J(\mathcal{T},\mathcal{S}))^{\mathfrak{A}} = 
(J(\mathcal{S},\mathcal{T}))^{\mathfrak{A}} = \mathcal{S}^{\mathfrak{A}}$
and $(J(\mathcal{T},\mathcal{T}))^{\mathfrak{A}}
= (\{\epsilon\},0)$. We 
call $J$ the \textbf{join} operator.\smallskip
\item[$\exists$\, )]
If $\mathit{ar}(\mathcal{T})
= k \geq 1$, we let
%
%
%
%
$(\exists(\mathcal{T}))^{\mathfrak{A}}$ be the AD-relation

\medskip

$
\bigl(\{(a_1, \dots , a_{k-1})\, |\, (a_1, \dots , a_k)\in\mathcal{T}^\mathfrak{A}\\
\text{ }\hspace{4cm}\text{ for some }a_k\in A\, \},k-1\bigr)$

\medskip

%


%
where $(a_1, \dots , a_{k-1})$ is the $(k-1)$-tuple
obtained by removing (i.e., 
projecting away) the last element of $(a_1, \dots , a_k)$.
When $\mathit{ar}(\mathcal{T}) = 0$,
then $(\exists(\mathcal{T}))^{\mathfrak{A}} = \mathcal{T}^{\mathfrak{A}}$.
We call $\exists$ is the \textbf{existence} or
\textbf{projection} operator. 

\smallskip

\medskip

\end{enumerate}

We denote this
algebra by $\mathit{\mathrm{GRA}}(e,p,s,I,\neg,J,\exists)$
where $\mathit{\mathrm{GRA}}$ stands for \textbf{general relation
algebra}. A set $\{f_1,\dots, f_k\}$ of operators
defines the general
relation algebra $\mathit{\mathrm{GRA}}(f_1,\dots , f_k)$; we shall
define various such systems below.
In this paper---only to simplify notation---we 
write $\mathit{\mathrm{GRA}}$
for $\mathit{\mathrm{GRA}}(e,p,s,I,\neg,J,\exists)$.
We identify $\mathrm{GRA}(f_1,\dots, f_k)$ with the
set of $\hat{\tau}$-terms of this algebra, where $\hat{\tau}$ is
the full relational vocabulary.
On the logic side, we similarly identify FO with the set of $\hat{\tau}$-formulas.

Let $\mathcal{G}$ be some set of terms of
some general relation algebra $\mathrm{GRA}(f_1,\dots, f_k)$. Formally,
the satisfiability problem for $\mathcal{G}$
takes as input a term $\mathcal{T}\in\mathcal{G}$   
and returns `\emph{yes}' iff there exists a model $\mathfrak{A}$
such that $\mathcal{T}^{\mathfrak{A}}$ is not the 
empty AD-relation of arity $\mathit{ar}(\mathcal{T})$.

An $\mathrm{FO}$-formula $\varphi$ and term $\mathcal{T}$
are \textbf{equivalent}
if $\varphi^{\mathfrak{A}} = \mathcal{T}^{\mathfrak{A}}$
for every $\tau$-model $\mathfrak{A}$ (where $\tau$ is an
arbitrary vocabulary that is large enough so
that $\varphi$ is a $\tau$-formula
and $\mathcal{T}$ a $\tau$-term). For example, the formula $R(v_1,v_2)$ is
equivalent to $R$, while $R(v_2,v_1)\wedge (P(v_1)
\vee \neg P(v_1))$ is equivalent to $sR$. Note 
that under our definition, $R(v_3, v_6)$ and $R(v_1,v_2)$
are both equivalent to the term $R$ while the formulas are
not equivalent to each other.
This causes no ambiguities as long as
we use the terminology carefully.
Also, $R(v_1,v_2) \wedge v_3 = v_3$ is \emph{not} 
equivalent to the term $R$ as it defines a
ternary rather than a binary relation. Furthermore, recall that in our setting, 
the formula $T(v_1,v_1,v_2)$ defines a binary relation and $v_8=v_8$ a
unary relation.

In the investigations below, it is useful to remember how the use of the operator $p$ is
reflected to corresponding $\mathrm{FO}$-formulas:
if $\mathit{rel}\bigl(R^{\mathfrak{A}}\bigr)
= \{(a,b,c,d)\} = \mathit{rel}
\bigl((Rxyzu)^{\mathfrak{A}}\bigr)$,
then $\mathit{rel}\bigl((pR)^{\mathfrak{A}}\bigr)
= \{(d,a,b,c)\} = \mathit{rel}
\bigl((Ryzux)^{\mathfrak{A}}\bigr)$, so the
tuple $(a,b,c,d)$ has its last element 
moved to the beginning of the tuple, while the formula $Rxyzu$
has the first variable $x$ moved to the end of
the tuple of variables.
It is also useful to 
understand how the operator $I$ works.
Now, if $\mathit{rel}\bigl(R^{\mathfrak{A}}\bigr)
= \{(a,b,c,d)\} = \mathit{rel}
\bigl((Rxyzu)^{\mathfrak{A}}\bigr)$, then
$$
\mathit{rel}\bigl((IR)^{\mathfrak{A}}\bigr)=
\begin{cases}
\{(a,b,c,c)\} \text{ if } c=d,\\ 
\varnothing \text{ otherwise},
\end{cases}
$$
so clearly the term $IR$ is equivalent to $Rxyzz$
which is obtained
from $Rxyzu$ by the variable substitution
that replaces $u$ with $z$.

Let $S_1$ be a
set of terms of our algebra and $S_2$ a set of FO-formulas.
We call $S_1$ and $S_2$ 
\textbf{equiexpressive} if
each $\mathcal{T}\in S_1$ has an equivalent formula in $S_2$ and
each $\varphi\in S_2$ an equivalent term in $S_1$.
The sets $S_1$ and $S_2$ are called \textbf{sententially equiexpressive} if
each \emph{sentence} $\varphi\in S_2$ has an equivalent term in $S_1$
and each term $\mathcal{T}\in S_1$ of arity $0$ has an
equivalent sentence in $S_2$.

\begin{theorem}
$\mathrm{FO}$ and $\mathrm{GRA}$ are
equiexpressive.\label{equiexpressivitytheorem}
\end{theorem}
\begin{proof}
Let us find an equivalent term for an $\mathrm{FO}$-formula $\varphi$.
Consider first the cases 
where $\varphi$ is one of the equality atoms $x=x$, $x=y$.
Then the corresponding
terms are, respectively, $I e$ and $e$.

Assume then that $\varphi$ is $R(v_{i_1},\dots , v_{i_k})$
for $k\geq 0$. Suppose 
first that no
variable symbol 
gets repeated in the
tuple $(v_{i_1},\dots , v_{i_k})$ and that $i_1 < \dots < i_k$.
Then the term $R$ is equivalent to $\varphi$.
We then consider the cases where $(v_{i_1},\dots , v_{i_k})$
may have repetitions and the variables may not be
linearly ordered (i.e., $i_1 < \dots < i_k$ does not
necessarily hold). We first observe that we can
permute any relation in every possible way by
using the operators $p$ and $s$; for the sake of
completeness, we present here the following steps that
prove this claim:

\begin{itemize}
\item
Consider a tuple $(a_1,\dots, a_i,
\dots , a_{k})$ of the relation $R^{\mathfrak{A}}$ in a
model $\mathfrak{A}$. Now, we can move the
element $a_i$ an arbitrary 
number $n$ of steps to the left (while
keeping the rest of 
the tuple otherwise in the same order) by doing the following:
\begin{enumerate}
\item
Repeatedly apply $p$ to the
term $R$, making $a_i$ 
the rightmost element of the tuple.
\item
Apply then the \emph{composed}
function $p s$ (so $s$ 
first and then $p$) precisely $n$ times.
\item
Apply $p$ repeatedly to put
the tuple
into
the ultimate desired order.
\end{enumerate}
\item
Moving $a_i$ to the right is similar. Intuitively, we
keep moving $a_i$ to the \emph{left} and
continue even when it has gone past the 
leftmost element of the original tuple. Formally, we can move $a_i$ by $n$
steps to the right by performing the above three
steps so that in step $2$, we apply the composed 
function $ps$ exactly $k - n- 1$ times.
\end{itemize}
\noindent
This shows that we can move an arbitrary element anywhere in
the tuple, and thereby it is clear that with $p$ and $s$ we
can permute a relation in all possible ways.

Since we indeed can permute tuples without
restrictions, we can
also deal with the possible 
repetitions of
variables in $R(v_{i_1},\dots , v_{i_k})$.
Indeed, we can bring any two
elements to the right end of a
tuple and then use $I$. We discussed this
phenomenon already above, but for extra clarity, we
once more illustrate the issue by providing a
related, concrete example. So let us consider the
formula $R(v_1,v_2,v_1)$ (which defines a
\emph{binary} relation). We observe that
$R(v_1,v_2,v_1)$ is equivalent to 
the term $p I p p (R)$, so we first use $p$
twice to permute $R$, then we use $I$ to identify
coordinates, and
finally we use $p$ once more.
%
%
%

So, to sum up, we permute tuples by $p$ and $s$ and we
use $I$ for identifying variables.
Therefore, using $p,s,I$, we can find an
equivalent term for 
every quantifier-free 
formula $R(v_{i_1},\dots , v_{i_k})$.

Now suppose we have equivalent terms $\mathcal{S}$ 
and $\mathcal{T}$ for 
formulas $\varphi$ and $\psi$, respectively.
We will discuss how to
translate $\neg\varphi$, $\varphi\wedge\psi$
and $\exists   v_i    \varphi$. Firstly, clearly $\neg\varphi$
can be translated to $\neg\mathcal{S}$.
Translating $\varphi\wedge\psi$ is done in 
two steps. Suppose $\varphi$ and $\psi$ have,
respectively, the free variables $v_{i_1},\dots , v_{i_k}$
and $v_{j_1},\dots , v_{j_{\ell}}$. We first write the term $J(\mathcal{S},\mathcal{T})$
which is equivalent to $\chi(v_1,\dots,
v_{k+\ell}) := \varphi(v_{1},\dots , v_{k})
\wedge \psi(v_{k+1},\dots , v_{k+\ell})$; note here the
new lists of variables. We then deal
with the possible
overlap in the original
sets $\{v_{i_l},\dots , v_{i_k}\}$
and $\{v_{j_l},\dots , v_{j_\ell}\}$ of
variables of $\varphi$ and $\psi$. This is
done by repeatedly
applying $p$, $s$ and $I$ to $J(\mathcal{S},\mathcal{T})$ in the very
same way as used above when dealing with atomic formulas. 
Indeed, we above observed that we can arbitrarily permute
relations and identify variables by using $p,s,I$.

Finally, translating $\exists v_i \varphi$ is 
easy. We first repeatedly apply $p$ to the term $\mathcal{S}$
corresponding to $\varphi$ to bring the element to be
projected away to the right end of the tuple.
Then we use $\exists$. After this we again use $p$
repeatedly to put the term into the final wanted form.

Translating terms to
equivalent $\mathrm{FO}$-formulas is straightforward.
\end{proof}

We obtain the following characterization of atomic 
formulas as a corollary.

\begin{corollary}\label{atomcorollary}
$\mathrm{GRA}(p,s,I)$ is equiexpressive
with the set of relational FO-atoms, and $\mathrm{GRA}(e,p,s,I)$ is
equiexpressive with the set of FO-atoms. 
\end{corollary}

\begin{proof}
We observed in the proof of Theorem \ref{equiexpressivitytheorem} that
all relational FO-atoms can be expressed in terms of $p,s,I$. The converse fact that 
every term $\mathrm{GRA}(p,s,I)$ expresses a relational atom is justified by
the following observations. Firstly, $p$ and $s$ just permute relational tuples.  
And secondly, $I$ has the following effect: if $\mathcal{T}$ and $R(v_{i_1},\dots , v_{i_k})$ are
equivalent, then so are $I(\mathcal{T})$
and $R(v_{i_1}',\dots , v_{i_k}')$, where $(v_{i_1}',\dots , v_{i_k}')$ is obtained
from $(v_{i_1},\dots , v_{i_k})$ by replacing the occurrences of the
variable with the greatest subindex $i_j$ by the variable with the
second greatest subindex.
The claim for $\mathrm{GRA}(e,p,s,I)$ is now also clear, recalling, in particular,
that $x= x$ and $x= y$ are equivalent to the terms $I e$ and $e$, respectively.
\end{proof}

The $\mathrm{FO}$-equivalent algebra $\mathrm{GRA} = 
\mathrm{GRA}(e,p,s,I,\neg,J,\exists)$ is only one of many
interesting related systems. Defining alternative algebras 
equiexpressive with $\mathrm{FO}$ is surely a relevant option, but it is 
also interesting to consider
weaker, stronger and orthogonal systems.
We next give a general definition that enables
classifying all such algebras in a systematic way.
In the definition, $\mathrm{AD}_A$ is
the set of all $\mathrm{AD}$-relations (of
every arity) over $A$.
If $T_1,\dots , T_k$
are AD-relations over a set $A$, then $(A,T_1,\dots , T_k)$ is
called an \textbf{AD-structure}. A bijection $g:A\rightarrow B$ is an
isomorphism between AD-structures $(A,T_1,\dots , T_k)$ and $(B,S_1,\dots , S_k)$
if $\mathit{ar}(T_i) = \mathit{ar}(S_i)$ for each $i$ 
and $g$ is an ordinary isomorphism
between $(A,\mathit{rel}(T_1),\dots , \mathit{rel}(T_k))$
and  $(B,\mathit{rel}(S_1),\dots , \mathit{rel}(S_k))$.

\begin{definition}\label{relationoperatordefinition}
A $k$-ary \textbf{relation operator} $f$ is a map that outputs, 
given an arbitrary set $A$, a $k$-ary
function $f^A : (\mathrm{AD}_A)^k \rightarrow \mathrm{AD}_A$. 
The operator $f$ is isomorphism
invariant in the sense that if the $\mathrm{AD}$-structures
$(A,T_1,\dots , T_k)$ and $(B,S_1,\dots , S_k)$ are
isomorphic via $g:A\rightarrow B$,
then $(A, f^A(T_1,\dots , T_k))$ and $(B, f^B(S_1,\dots , S_k))$ 
are, likewise, isomorphic via $g$.
\end{definition}
%
%
%
%
%
%
An \textbf{arity-regular relation operator} is a relation operator with
the property that the arity of the output AD-relation depends only on
the sequence of arities of the input AD-relations.

To illustrate the notion of a relation operator, let us consider some
concrete examples.
Suppose $\mathcal{T}$ and $\mathcal{S}$ are
both of arity $k$. We define

\begin{itemize}
\item
$(\mathcal{T} \cup \mathcal{S})^\mathfrak{A}
= (\mathit{rel}(\mathcal{T}^\mathfrak{A})
\cup \mathit{rel}(\mathcal{S}^\mathfrak{A}), k),$ 
\item
$(\mathcal{T} \cap \mathcal{S})^\mathfrak{A}
= (\mathit{rel}(\mathcal{T}^\mathfrak{A})
\cap \mathit{rel}(\mathcal{S}^\mathfrak{A}), k),$
\item
$(\mathcal{T} \setminus \mathcal{S})^\mathfrak{A}\, 
=\, (\mathit{rel}(\mathcal{T}^\mathfrak{A})
\setminus \mathit{rel}(\mathcal{S}^\mathfrak{A}), k),$
\end{itemize}
and if $\mathcal{T}$ and $\mathcal{S}$ have 
different arities, then $\cap$ and $\cup$
return $(\varnothing,0)$ and $\setminus$
returns $\mathcal{T}^{\mathfrak{A}}$.
Suppose then that $\mathcal{T}$ and $\mathcal{S}$ 
have arities $k$ and $\ell$, respectively.
%
Calling $m:= \mathit{max}\{k,\ell\}$, we let

\begin{align*}
(\mathcal{T}\, \Dot{\cap}\,
\mathcal{S})^\mathfrak{A}
&= \bigl(\{(a_1,\dots ,a_m)
\mid (a_{m-k+1},\dots ,a_m) \in \mathcal{T}^\mathfrak{A}\\
&\text{ }\hspace{1.7cm}\text{ and } (a_{m-\ell + 1},\dots ,a_{m}) \in
\mathcal{S}^\mathfrak{A}\},\, m \bigr),
\end{align*}

\smallskip

\noindent
so intuitively, the tuples overlap on some
suffix of $(a_1,\dots , a_m)$; note here 
that when $k$ or $\ell$ is zero,
then $(a_{m+1},a_m)$ denotes the empty tuple $\epsilon$.
Now for example the formula $R(x,y) \land P(y)$ is
equivalent to $R\, \Dot{\cap}\, P$
and the formula $R(x,y) \land P(x)$ to $s(sR\ \Dot{\cap}\, P)$.
We call $\Dot\cap$ the $\textbf{suffix intersection}$.


In the next section we prove that the guarded 
fragment $\mathrm{GF}$ is sententially equivalent to $\mathrm{GRA}(e,p,s,\setminus,\Dot{\cap},\exists)$.
We note that in \cite{semijoinrelationalgebra,
hirsch}, the authors define
Codd-style relational algebras \emph{with  inherently infinite signatures}, and they then
prove the algebras to be sententially equiexpressive with $\mathrm{GF}$.
The system in \cite{semijoinrelationalgebra} uses, e.g., a \emph{semijoin operator}, which is
essentially a join operation but employs also a conjunction of identity atoms as part of the input to it.
The algebra of \cite{hirsch} employs, e.g., a ternary join operator where one of the inputs essentially
acts as a guard. Both algebras have an implicit 
access to variables via the infinite signatures in
the usual way of Codd-style systems.
%

Importantly, the proofs of the characterizations in \cite{semijoinrelationalgebra, hirsch} 
differ considerably from our corresponding argument, the translations
from algebra to logic being
inherently exponential in 
\cite{semijoinrelationalgebra}
and \cite{hirsch}.
We carefully develop techniques that allow us to give a polynomial translation from $\mathrm{GRA}(e,p,s,\backslash,\Dot{\cap},\exists)$ to $\mathrm{GF}$, which in turn
%
allow us to prove a \textsc{2ExpTime}
upper bound for the satisfiability
problem of the algebra $\mathrm{GRA}(e,p,s,\backslash,\Dot{\cap},\exists)$, the same as
that for $\mathrm{GF}$. Since we 
will also give a polynomial translation
from $\mathrm{GF}$ to $\mathrm{GRA}(e,p,s,\backslash,\Dot{\cap},\exists)$, it
follows that the satisfiability problem for the
algebra is \textsc{2ExpTime}-complete.
Furthermore, the algebra $\mathrm{GRA}(e,p,s,\backslash,\Dot{\cap},\exists)$ is a
genuinely variable-free system that indeed has a finite signature and a simple set of operators.

In addition to the guarded fragment, we
can also prove the following characterization for two-variable logic.

\begin{theorem}\label{fotwocharacterization}
$\mathrm{FO}^2$ and 
$\mathrm{GRA}(e,s,\neg,\Dot\cap,\exists)$
are sententially equiexpressive over vocabularies with at
most binary relations.
\end{theorem}

\begin{proof}

The algebra $\mathrm{GRA}(e,s,\neg,\Dot\cap,\exists)$ with at most binary 
relation symbols clearly contains
only terms of arity at most two. Thus it is easy to translate the terms into $\mathrm{FO}^2$.
%

We then consider the converse translation. 
We assume that $\mathrm{FO}^2$ is
built using $\neg,\land$ and $\exists$ and treat other
connectives and $\forall$ as abbreviations in the usual way.

Now, let $\varphi \in \mathrm{FO}^2$
be a \emph{sentence} with at
most binary relations, and let $x$ and $y$ be the two variables that occur in $\varphi$. Note indeed that $\varphi$ is a
sentence, not an open formula.
We first convert $\varphi$ into a
sentence that does not contain any
subformulas of type $\psi(x) \wedge \chi(y)$ (or of
type $\psi(x) \vee \chi(y)$) as follows. 
Consider any subformula $\exists x\, \eta(x,y)$ where $\eta(x,y)$ is
quantifier-free. Put $\eta$ into disjunctive normal form and
distribute $\exists x$ over the disjunctions. Then
distribute $\exists x$ also over the 
over conjunctions as follows. Consider a
conjunction $\alpha_i(x,y) \wedge \beta_i(y) \wedge \gamma_i$
where each of $\alpha_i$, $\beta_i$, $\gamma_i$ are
conjunctions of literals;
the formula $\gamma_i$ contains the nullary relation symbols and
$\alpha_i(x,y)$ contains the literals of 
type $\pi(x,y)$ and $\pi'(x)$.
We distribute $\exists x$ into $\alpha_i(x,y) \wedge \beta_i(y) \wedge \gamma_i$ so 
that we obtain the formula $\exists x \alpha_i(x,y)
\wedge \beta_i(y) \wedge \gamma_i$.
Thereby the formula $\exists x\, \eta(x,y)$ gets
modified into the formula
$$\bigvee_{i=1}^{n} (\exists x \alpha_i(x,y)
\wedge \beta_i(y) \wedge \gamma_i)$$ which is of 
the right form and does not have $x$ as a 
free variable.
Next we can repeat this process for other existential
quantifiers in the formula
(by treating the subformulas with one free variable in the
way that atoms with one free variable were treated in
the translation step for $\eta(x,y)$ described above).
Having started from the \emph{sentence} $\varphi$, we
ultimately get a sentence that does not have
subformulas of the form $\psi(x) \wedge \chi(y)$ or
of the form $\psi(x) \vee \chi(y)$ but is nevertheless
equivalent to $\varphi$.

Next we translate an arbitrary 
sentence $\varphi \in \mathrm{FO}^2$ that
satisfies the above condition to an equivalent term.
We let $v\in\{x,y\}$ denote a generic variable.

Atoms of the form $P(v)$ (respectively $v=v$) 
translate to $P$ (respectively $\exists e$).
Relation symbols of arity $0$ translate to 
themselves and

\begin{enumerate}
    \item $R(x,y)$ translates to $R$,
    \item $R(y,x)$ translates to $sR$,
    \item $R(v,v)$ translates to $\exists (R\, \Dot{\cap}\, e)$,
    \item $x = y$ and $y=x$ translate to the term $e$.
\end{enumerate}

Now suppose we have
translated $\psi$ to $\mathcal{T}$.
Then $\neg \psi$
translates to $\neg \mathcal{T}$.  If $\psi$ has
one free variable $v$, then $\exists v \psi$
translates to $\exists\mathcal{T}$. If $\psi$
has two free variables, then we
either translate $\exists v \psi$
to $\exists\mathcal{T}$ when $v$ is $y$
and to $\exists s\mathcal{T}$ when $v$ is $x$.

Consider now a formula $\psi \wedge \chi$ and suppose that we have translated $\psi$ to $\mathcal{T}$
and $\chi$ to $\mathcal{S}$. If at 
least one of $\psi$ and $\chi$ is a 
sentence, we
translate $\psi \wedge \chi$ to $(\mathcal{T}\, 
\Dot{\cap}\, \mathcal{S})$.
Otherwise, due to the form of the
sentence $\varphi$ to be translated, we 
have $$\Free(\psi) \cap \Free(\chi) \neq \varnothing.$$
Now $\psi(x,y) \wedge \chi(x,y)$, $\psi(y) \wedge \chi(x,y)$ and $\psi(x,y) \wedge \chi(y)$ are all translated to $\mathcal{T}\, \Dot{\cap}\, \mathcal{S}$, 
while $\psi(x,y)\wedge \chi(x)$ and $\psi(x) \wedge \chi(x,y)$ are translated to $s(s\mathcal{T}\, \Dot{\cap}\, \mathcal{S})$ and $s(\mathcal{T}\, \Dot{\cap}\, s\mathcal{S})$, respectively. 
\end{proof}

We note that limiting our algebraic characterizations of $\mathrm{GF}$ 
and $\mathrm{FO}^2$ to \emph{sentential}
equiexpressivity is a 
choice based on the relative elegance of the results. We shall give related characterizations
without the limitation in the full version of the paper.
We note, however, that 
sentential equiexpressivity suffices for the almost all
practical scenarios.

Now, let $\mathrm{GRA}_2(e,s,\neg,\Dot{\cap},\exists)$ denote the
terms of $\mathrm{GRA}(e,s,\neg,\Dot{\cap},\exists)$ that 
use at most binary relation symbols; there are 
no restrictions on term arity, although it is easy to 
see that at most binary terms arise.
The proof of Theorem \ref{fotwocharacterization}
gives a translation from $\mathrm{FO}^2$-sentences (with at 
most binary symbols) 
to $\mathrm{GRA}_2(e,s,\neg,\Dot{\cap},\exists)$. However,
that translation is not polynomial, and thus it is not
immediately clear if we get a \textsc{NexpTime} lower bound for the
the satisfiability
problem of the system $\mathrm{GRA}_2(e,s,\neg,\Dot{\cap},\exists)$.
Nevertheless, we can prove the following Theorem. 

\begin{theorem}
The satisfiability 
problem of\ \ $\mathrm{GRA}_2(e,s,\neg,\Dot{\cap},\exists)$\ is \textsc{NexpTime}-complete.
\end{theorem}

\begin{proof}
The upper bound follows from the fact that $\mathrm{GRA}_2(e,s,\neg,\Dot{\cap},\exists)$ 
translates easily into $\mathrm{FO}^2$ in polynomial time
and $\mathrm{FO}^2$ is
well known to have a \textsc{NexpTime}-complete 
satisfiability problem.

We then consider the lower bound. For this, we use fluted logic (FL),
defined in the appendix (Definition \ref{fluteddefinition}).
The proof of Proposition \ref{flutedcharacterizationappendix} in the appendix shows 
that the \emph{two-variable fragment of fluted logic} over
vocabularies with at most binary 
relation symbols translates in polynomial
time into $\mathrm{GRA}_2(e,s,\neg,\Dot{\cap},\exists)$.
In \cite{flutedlidiatendera}, it is established that
the two-variable fragment of fluted logic is \textsc{NexpTime}-complete,
and at most binary relations suffice for the lower bound there.
%
%
%
%
\end{proof}
%

%

%
%

We then briefly consider fluted logic ($\mathrm{FL}$),
mentioned in the above proof.
The logic $\mathrm{FL}$ is a decidable fragment of $\mathrm{FO}$ that has recently received increased
attention in the research on first-order fragments. Definition \ref{fluteddefinition} in
the appendix gives a formal definition of the system. For the history of $\mathrm{FL}$, we
recommend the introduction of the article \cite{flutedlidiatendera}.

Now, it is straightforward to show (see Proposition \ref{flutedcharacterizationappendix} in
the appendix)
that fluted logic is
equiexpressive with $\mathrm{GRA}(\neg,\Dot\cap,\exists)$.
By comparing the algebraic characterizations, we
observe $\mathrm{FO}^2$ and fluted logic are very
interestingly and intimately related, and the full system
$\mathrm{GRA}(e,s,\neg,\Dot\cap,\exists)$ obviously
contains both fluted logic and $\mathrm{FO}^2$. Note also 
the close relationship of these systems to 
the algebra $$\mathrm{GRA}(e,p,s,
\setminus,\Dot\cap,\exists)$$ for GF.
These connections clearly demonstrate how the
algebraic approach can 
elucidate the relationships between seemingly different
kinds of fragments of FO. Indeed, $\mathrm{FO}^2$, FL 
and GF seem much more closely related than one might 
first suspect. By the algebraic characterizations, the 
three logics become associated with three simple 
finite collections of operators (these being 
the corresponding algebraic signatures),
and there are nice elucidating links between the collections.

The logic $\mathrm{UNFO}$ is a well-established decidable fragment of $\mathrm{FO}$ that enjoys many of the desirable properties that modal logics have \cite{segoufin}. Roughly speaking, its syntax is obtained from that of $\mathrm{FO}$ by restricting the use of negation only to formulas that have at most one free variable. To characterize $\mathrm{UNFO}$, we will need to introduce two additional relation operators. Suppose that $\mathcal{T}$ and $\mathcal{S}$ are terms of arity $k$ AND $\ell$ respectively. We define $(\overline{J}(\mathcal{T},\mathcal{S}))^\mathfrak{A}$ to be equal to 

\medskip

$(\{(a_1,\dots,a_k,b_1,\dots,b_\ell) \mid (a_1,\dots,a_k) \in \mathcal{T}^\mathfrak{A} \text{ or } (b_1,\dots,b_\ell) \in \mathcal{S}^\mathfrak{A}\}, k + \ell).$

\medskip

\noindent
Thus $\overline{J}$ is the \emph{dual} of $J$. If $k\leq 1$, we define $(\neg_1(\mathcal{T}))^\mathfrak{A} = (\neg (\mathcal{T}))^\mathfrak{A}$, and otherwise $(\neg_1(\mathcal{T}))^\mathfrak{A} = \bot_0$. We call $\neg_1$ the \textbf{one-dimensional negation}. In Appendix \ref{unfoappendix} $\mathrm{UNFO}$ is equiexpressive with $\mathrm{GRA}(e,p,s,I,\neg_1,J,\overline{J},\exists)$.

We point out yet another natural fragment inspired by our algebraic approach, namely the algebra $\mathrm{GRA}(e,s,\backslash,\Dot{\cap},\exists)$. This algebra is interesting because, for example, it contains the guarded $\mathrm{FO}^2$ and the guarded $\mathrm{FL}$ on the level of sentences. 
\begin{theorem}\label{theorem:unfo-fluted-complexity}
    The satisfiability problem of $\mathrm{GRA}(e,s,\backslash,\Dot{\cap},\exists)$ is \textsc{ExpTime}-complete.
\end{theorem}
\begin{proof}
See Appendix 
\ref{appendix:unfo-fluted-complexity}
\end{proof}

Going beyond $\mathrm{FO}$ is a key future research direction.
There are many relevant possibilities.
Define the \textbf{equicardinality operator} $H$
such that $(H(\mathcal{T},\mathcal{S}))^{\mathfrak{A}} = \top_0$
if $\mathit{rel}(\mathcal{T}^{\mathfrak{A}})$
and $\mathit{rel}(\mathcal{S}^{\mathfrak{A}})$ have the 
same (possibly infinite cardinal) number of tuples; else
output $\bot_0$. \emph{No restrictions on the input
relation arities are imposed}.
Adding related quantifiers (e.g., the 
H\"{a}rtig quantifier) to quite weak
$\mathrm{FO}$-fragments is
known to lead to
undecidability. Nevertheless, interesting related decidable systems can be defined. Also, the transitive closure operator is interesting and relevant in its various possible forms.

\section{An algebra for the guarded fragment}

In this section we consider
$\mathrm{GRA}(e,p,s,\setminus, \dot\cap , \exists)$
and show that it is sententially equiexpressive with $\mathrm{GF}$.
Recall that $\mathrm{GF}$ is the logic that has all
atoms $R(x_1,\dots,x_k)$, $x=y$ and $x=x$, is closed under $\neg$
and $\wedge$, but existential quantification is
restricted to patterns $\exists x_1\dots \exists x_k(\alpha\, \wedge\, \psi)$
where $\alpha$ is an atomic formula (a guard)
having (at least) all the free variables of $\psi\in \mathrm{GF}$.



Going beyond $\mathrm{FO}$ is a key future research direction.
For the sake of brevity, we 
mention here only one of many relevant operators to be investigated.
Define the \textbf{equicardinality operator} $H$
such that $(H(\mathcal{T},\mathcal{S}))^{\mathfrak{A}} = \top_0$
if $\mathit{rel}(\mathcal{T}^{\mathfrak{A}})$
and $\mathit{rel}(\mathcal{S}^{\mathfrak{A}})$ have the 
same (possibly infinite cardinal) number of tuples; else
output $\bot_0$. \emph{No restrictions on the input
relation arities are imposed}.
Adding related quantifiers (e.g., the 
H\"{a}rtig quantifier) to quite weak
$\mathrm{FO}$-fragments is
known to lead to
undecidability. Nevertheless, we have already 
found natural decidable systems containing
the operator $H$---to be
discussed in the full version---that can be
defined via suitably chosen sets of 
operators used.

We start by
defining the notion of a \textbf{term guard} of a 
term $\mathcal{T}$. Term guards are a central concept in
our proof. The term guard of a term $\mathcal{T}$ of $\mathrm{GRA}(e,p,s,\backslash,\Dot{\cap},\exists)$ is a
tuple $(\mathcal{S},(i_1,\dots ,i_{k}))$, 
where $k = ar(\mathcal{T}) \leq \mathit{ar}(\mathcal{S})$,
with the following properties.

\medskip

\begin{enumerate}
    \item 
    $\mathcal{S}$ is either $e$ or a relation symbol occurring in $\mathcal{T}$.
    \item The tuple $(i_1,\dots , i_k)$ consists of pairwise distinct integers
              $i_j$ such that $1 \leq i_j \leq \mathit{ar}(\mathcal{S})$.
    \item
    For every model $\mathfrak{A}$ and every tuple $(a_1,.\dots , a_k)\in \mathcal{T}^\mathfrak{A}$,
    there exists a tuple $(b_1,\dots , b_{\mathit{ar}(\mathcal{S})}) \in \mathcal{S}^\mathfrak{A}$ so
    that $(a_1,\dots , a_k) = (b_{i_1}, \dots , b_{i_k})$.
%
%
\end{enumerate}

\medskip

The intuition is
that the term guard $(\mathcal{S},(i_1,\dots , i_k))$ of $\mathcal{T}$ gives an atomic term $\mathcal{S}$
and a list $(i_1,\dots , i_k)$ of coordinate positions (of tuples of $\mathcal{S}^{\mathfrak{A}})$ 
that guard the tuples of $\mathcal{T}^{\mathfrak{A}}$.
The remaining $m-k$ coordinate positions of the tuples of $\mathcal{S}^{\mathfrak{A}}$ are intuitively
non-guarding.

The following lemma 
will be used below when translating algebraic 
terms to formulas of the guarded fragment.

\begin{lemma}\label{termguardlemma}
Every term $\mathcal{T}\in \mathrm{GRA}(e,p,s,\setminus,\Dot{\cap},\exists)$ has a term
guard. Furthermore the term guard can be computed from $\mathcal{T}$ in polynomial time.
\end{lemma}
\begin{proof}
We will define inductively a mapping which maps each term $\mathcal{T}$
of the system $\mathrm{GRA}(e,p,s,\backslash, \Dot{\cap},\exists)$ to a term guard for $\mathcal{T}$.
We start by defining that $e$ will be mapped to $(e,(1,2))$ and that
every relational symbol $R$ will be mapped to $$(R,(1,\dots ,ar(R))).$$

Suppose then that we have mapped a $k$-ary term $\mathcal{T}$ to the term guard $$(\mathcal{S},(i_1,\dots,i_k)).$$
%
Using the term guard $(\mathcal{S},(i_1,\dots,i_k))$ as a starting point, we will construct term
guards for the terms $p\mathcal{T}$ and $s\mathcal{T}$.
Firstly, term guard for $p\mathcal{T}$ will be $$(\mathcal{S},(i_k,i_1, \dots , i_{k-1})),$$ where we
have simply permuted the tuple $(i_1,\dots , i_k)$ with $p$. (Note that if $k\leq 1$, then
the permuted
tuple is the same as the original tuple, as $p$
leaves tuples of length up to $1$ untouched.)
Similarly, the term guard for $s\mathcal{T}$ will be $$(\mathcal{S},(i_1, \dots , i_{k-2} , i_k, i_{k-1})),$$ where
this time we have permuted the tuple $(i_1,\dots , i_k)$ with $s$. (Again if $k\leq 1$, the permuted 
tuple is the original tuple.)
%
%

The other cases are similar. Recall the assumption that we have
mapped a $k$-ary term $\mathcal{T}$ to the term
guard $(\mathcal{S},(i_1,\dots,i_k))$, and
suppose further that an $\ell$-ary term $\mathcal{P}$ has been mapped to the
term guard $(\mathcal{S}',(j_1,\dots ,j_\ell))$. If $k\geq \ell$, 
then $\mathcal{T}\, \Dot{\cap}\, \mathcal{P}$ 
will be mapped to $(\mathcal{S},(i_1,\dots ,i_k))$,
and if $k < \ell$, then 
$\mathcal{T}\, \Dot{\cap}\, \mathcal{P}$ 
will be mapped
to $(\mathcal{S}',(j_1,\dots ,j_\ell))$.
Independently of how the arities $k$ and $\ell$ are 
related, $\mathcal{T}\backslash \mathcal{P}$
will always be mapped to $(\mathcal{S},(i_1,\dots ,i_k))$.
(Recall that if the arities of the terms $\mathcal{Q}$ and $\mathcal{R}$ differ, then by
definition $\mathcal{Q}\backslash\mathcal{R}$ is
equivalent to $\mathcal{Q}$).
If $k\geq 1$, the term $\exists \mathcal{T}$ will be
mapped to $(\mathcal{S},(i_1,\dots ,i_{k-1}))$. 
If $k = 0$, the term $\exists \mathcal{T}$ 
simply maps to the same term guard as
the term $\mathcal{T}$.

This completes the definition of the mapping. Since the mapping is clearly computable in polynomial time, the claim follows.
\end{proof}

We will also make use of the following lemma which implies that we
can use $I$ in $\mathrm{GRA}(e,p,s,\setminus,\Dot\cap,\exists)$.

\begin{lemma}\label{identityeliminationlemma}
The operator $I$ can be expressed with $e$ and $\Dot\cap$ as follows. If $\mathit{ar}(
\mathcal{T}) > 1$, then $I\mathcal{T}$ is
equivalent to $\exists(\mathcal{T}\, \Dot{\cap}\, e)$,
and if $\mathit{ar}(
\mathcal{T}) \leq 1$, then $I \mathcal{T}$ is
equivalent to $\mathcal{T}$.
%
\end{lemma}
\begin{proof}
Immediate.
%
\end{proof}

%
%
%

We can now prove the following.

\begin{theorem}
$\mathrm{GRA}(e,p,s,\setminus,\dot\cap,\exists)$ and
$\mathrm{GF}$ are sententially equiexpressive.
\end{theorem}

\begin{proof}
We will first show that for every formula 
$\exists x_1\dots \exists x_k\, \psi$
of $\mathrm{GF}$, there exists an
equivalent term $\mathcal{T}$
of $\mathrm{GRA}(e,p,s,\setminus,\dot\cap,\exists)$.
Let us begin by 
showing this for a 
formula $\varphi := \exists x_1\dots \exists x_k\, \psi$
where $\psi$ is quantifier-free.
We assume $$\varphi = \exists x_1\dots \exists x_k( 
\alpha(y_1,\dots ,y_n) \wedge \beta(z_1,\dots ,z_m))$$
where $\alpha(y_1,\dots ,y_n)$ is an
atom and we have
$$\{z_1,\dots, z_m\}\subseteq \{y_1,\dots , y_n\}
\text{ and }\{x_1,\dots, x_k\}\subseteq \{y_1,\dots , y_n\}.$$

Now consider a conjunction $\alpha\wedge \rho$
where $\alpha = \alpha(y_1,\dots , y_n)$ is our guard
atom and $\rho$ an arbitrary
atom whose set of variables is a subset of $\{y_1,\dots , y_n \}$.
We call such a conjunction an \emph{$\alpha$-guarded atom}.
For each $\alpha$-guarded atom, we can find an
equivalent term as follows. First, by Corollary \ref{atomcorollary},
we can write a
term equivalent to any atomic FO-formula using $e,p,s,I$; 
note that we can
use $I$ in $\mathrm{GRA}(e,p,s,\setminus,\Dot\cap,\exists)$ by
Lemma \ref{identityeliminationlemma}.
Therefore we can find
terms $\mathcal{T}_{\alpha}$ and $\mathcal{T}_{\rho}$
equivalent to $\alpha$ and $\rho$, respectively.
Now, the term $\mathcal{T}\, \dot\cap\, \mathcal{S}$ is not
likely to be equivalent to $\alpha\wedge\rho$, as
the variables in $\alpha\wedge\rho$ can be unfavourably 
ordered instead of matching each other nicely.
However---recalling that $p$ and $s$ can be composed to produce arbitrary permutations---we first permute $\mathcal{T}_{\alpha}$ to
match $\mathcal{T}_{\rho}$ at the 
last coordinates of tuples, then we combine the terms with $\dot\cap$, and 
finally we permute the obtained term to the final desired form.
In this fashion we obtain a term for an arbitrary $\alpha$-guarded atom.

Now recall the formula $\alpha(y_1,\dots , y_n)\wedge\beta(z_1,\dots,z_m)$
from above. For each atom $\gamma$ in $\beta$,
let $\mathcal{T}_{\gamma}^{\alpha}$ denote the term 
equivalent to the $\alpha$-guarded atom
$\alpha\wedge\gamma$ formed from $\gamma$. The formula $\beta$ is a
Boolean combination composed from 
atoms by using $\neg$ and $\wedge$. We let $\mathcal{T}_{\beta}$
denote the term obtained from $\beta$ by replacing each 
atom $\gamma$ by the term $\mathcal{T}_{\gamma}^{\alpha}$,
each $\wedge$ by $\dot\cap$ and 
each $\neg$ by relative complementation 
with respect to $\mathcal{T}_{\alpha}$ (i.e., 
formulas $\neg\, \eta$ become
replaced by $\mathcal{T}_{\alpha}\setminus \eta^*$
where $\eta^*$ is the translation of the formula $\eta$).
It is easy to show that $\mathcal{T}_{\beta}$ is
equivalent to $\alpha(y_1,\dots , y_n)\wedge\beta(z_1,\dots,z_m)$. 
Thus we can clearly use $p$ and $\exists$ in a
suitable way to the term $\mathcal{T}_{\beta}$ to
get a term equivalent to the formula $\varphi = \exists x_1\dots
\exists x_k(\alpha(y_1,\dots ,y_n) \wedge \beta(z_1,\dots ,z_m))$.

Thus we managed to translate $\varphi$. To get the 
full translation, we mainly just keep repeating the procedure 
just described. The only difference is that above
the formula $\beta(z_1,\dots, z_m)$ was a Boolean 
combination of atoms, while now $\beta$ will also 
contain formulas of 
the form $\exists x_1\dots \exists x_{r}(\delta\wedge\eta)$ in
addition to atoms.
Proceeding by induction, we get a term 
equivalent to $\exists x_1\dots \exists x_{r}(\delta\wedge\eta)$ by
the induction hypothesis, and
otherwise we proceed exactly as described above.
This concludes the argument for translating formulas to terms.

Let us then consider how to translate terms into
equivalent formulas of $\mathrm{GF}$. The proof proceeds by induction.
Since $\mathrm{GF}$ is closed under Boolean operators,
the only non-trivial case is the translation of the projection operator $\exists$.
The hard part in this case is to ensure
that we can translate $\exists$ so that the resulting formula has a 
suitable guarding pattern (with a suitable guard atom) and
thereby belongs to $\mathrm{GF}$.

So suppose that we have translated $\mathcal{T}$ to $\psi(v_1,\dots ,v_k)$.
By Lemma \ref{termguardlemma}, we can find a
term guard $(\mathcal{S},(i_1,\dots,i_k))$ for $\mathcal{T}$.
By the definition of term guards, $\mathcal{S}$ is $e$ or
some relation symbol $R$.
We let $m$ denote the arity of $\mathcal{S}$,
and we let $\alpha(v_1,\dots , v_m)$ denote 
\begin{enumerate}
\item
the formula $R(v_1,\dots ,v_m)$ if $\mathcal{S}$ is a relation symbol,
\item
the formula $v_1 = v_2$ if $\mathcal{S} = e$ and 
therefore $m = 2$.
\end{enumerate}
%
%
%
%
%
Notice that $\{i_1,\dots , i_k\}
\subseteq \{1,\dots , m\}$ by the definition of
term guards.

Now define $$\chi(v_{i_1},\dots , v_{i_k})
:= \exists \overline{v} (\alpha(v_1,\dots,v_m) \land \psi(v_{i_1},\dots,v_{i_k})),$$ where $\overline{v}$
lists those variables from $\{v_1,\dots,v_m\}$ that are \emph{not}
included in $(v_{i_1},\dots,v_{i_k})$. Modifying $\chi(v_{i_1},\dots , v_{i_k})$ to 
the formula  $\chi(v_1,\dots , v_k)$ and recalling the definition of term guards, we
now observe that $\exists \mathcal{T}$ is
equivalent to $\exists v_k \chi(v_1,\dots , v_k)$.
%
%
%
%
%
%
%
%
\end{proof}

It is easy to see that our translations in the above
proof are computable in polynomial time. Since the
satisfiability problem for $\mathrm{GF}$ is \textsc{2ExpTime}-complete, we have the following corollary.

\begin{corollary}
    The satisfiability problem for $\mathrm{GRA}(e,p,s,\backslash,\Dot{\cap},\exists)$ is \textsc{2ExpTime}-complete.
\end{corollary}

%

\section{Decidable fragments of GRA}

In this section we identify
subsystems of $\mathrm{GRA} = \mathrm{GRA}(e,p,s,I,\neg,J,\exists)$ with a
decidable satisfiability problem. We concentrate on systems obtained by 
limiting to a subset of the operators involved. We show that
removing any of the operators $\neg,\exists,J,I$ leads to
decidability, and we also pinpoint the exact
complexity of each system. As a by-product of the work, we 
make observations about conjunctive queries (CQs)
and show \textsc{NP}-completeness of, e.g., $\mathrm{GRA}(\neg,J,\exists)$
and $\mathrm{GRA}(I,\neg,J)$. As a further by-product, we also give a characterization
for quantifier-free $\mathrm{FO}$.

Our first result concerns $\mathrm{GRA}$ with the
complementation operation $\neg$ removed.
All negation-free fragments of $\mathrm{FO}$
are trivially 
decidable---every formula being satisfiable---and thus so is $\mathrm{GRA}(e,p,s,I,J,\exists)$. Nevertheless, this system
has the following very interesting property concerning
conjunctive queries with equality, or CQEs.

\begin{proposition}\label{conj1}
$\mathrm{GRA}(e,p,s,I,J,\exists)$ is
equiexpressive with the set of\, $\mathrm{CQE}$s.
Also, the system $\mathrm{GRA}(p,s,I,J,\exists)$ is equiexpressive 
with the set of conjunctive queries $\mathrm{(CQ}$s$\mathrm{)}$. 
\end{proposition}

\begin{proof}
Analyzing the proof that $\mathrm{GRA}(e,p,s,I,\neg, J,\exists)$
and $\mathrm{FO}$ are equiexpressive, we see
that $\mathrm{GRA}(e,p,s,I,J,\exists)$ can
express every formula built from relational atoms and 
equality atoms with conjunctions and existential quantification.
Conversely, an easy induction on term structure establishes that
every term of the system $\mathrm{GRA}(e,p,s,I,J,\exists)$ is
expressible by a CQE.
The claim for $\mathrm{GRA}(p,s,I,J,\exists)$ follows similarly, noting that
equality atoms are used only to express the atoms $x= x$ and $x= y$ in
the proof of Theorem \ref{equiexpressivitytheorem}.
%
\end{proof}

We then study $\mathrm{GRA}$ without $\exists$.
We begin with the following expressivity characterization.

\begin{proposition}
$\mathrm{GRA}(e,p,s,I,\neg,J)$ is equiexpressive with 
the set of quantifier-free FO-formulas.
\end{proposition}

\begin{proof}
The fact that every term of $\mathrm{GRA}(e,p,s,I,\neg,J)$ translates to a
quantifier-free formula is seen by induction on terms. The non-trivial case is
the use of $I$. Here, similarly to what we already observed in the proof of
Corollary \ref{atomcorollary}, we note that if $\mathcal{T}$ and $\varphi$ are
equivalent, then so are $I(\mathcal{T})$
and the formula $\varphi'$ obtained from $\varphi$ by replacing the occurrences of the
free variable with the greatest subindex by the variable with the
second greatest subindex. (Here we must take care to avoid variable capture.)

The fact that quantifier-free FO translates into $\mathrm{GRA}(e,p,s,I,\neg,J)$ is
immediate by inspection of the proof of Theorem \ref{equiexpressivitytheorem}.
\end{proof}

The translations of the proposition are clearly polynomial,
whence we get the following corollary due to the folklore
result that the the satisfiability problem for quantifier-free FO is $\textsc{NP}$-complete.

\begin{corollary}\label{whatevercorollary}
The satisfiability problem for $\mathrm{GRA}(e,p , s,I,\neg, J)$ is \textsc{NP}-complete.
\end{corollary}

To sharpen the lower bound result, we also prove the following.

\begin{proposition}\label{whateverproposition}
The satisfiability problem for $\mathrm{GRA}(I,\neg, J)$ is \textsc{NP}-complete.
\end{proposition}
\begin{proof}
The upper bound follows from Corollary \ref{whatevercorollary}.
For the lower bound, we give a reduction from SAT. 
So, let $\varphi$ be a propositional logic formula and $\{p_1,\dots ,p_n\}$ the set of proposition symbols of $\varphi$. Let $\{P_1,\dots ,P_n\}$ be a set of unary relation symbols. We translate $\varphi$ to an equisatisfiable term $\mathcal{T}$ of $\mathrm{GRA}(I,\neg,J)$ as follows. Every $p_i$ translates to $P_i$. If $\psi$ is
translated to $\mathcal{T}$, then $\neg \psi$ is
translated to $\neg \mathcal{T}$. If $\psi$ translates to $\mathcal{T}$
and $\theta$ to $\mathcal{P}$, then $\psi \wedge \theta$ is translated to $J(\mathcal{T},\mathcal{P})$.
Let $\mathcal{T}(\varphi)$ be the resulting term obtained like this 
from $\varphi$,
and note that $\mathcal{T}(\varphi)$ is not yet
equisatisfiable with $\varphi$, as for
example $p_1 \wedge \neg p_1$
translates to a term equivalent to the formula $P_1(v_1) \wedge \neg P_1(v_2)$.
Thus we still need to express that all
variables $v_i$ are equal. This is easy to do with repeated use of the operator $I$.
\end{proof}

We then leave the investigation of the $\exists$-free 
system $\mathrm{GRA}(e,p , s,I,\neg, J)$ and
consider the join-free fragment of $\mathrm{GRA}$. This system
turns out to be interestingly tame, having a very low complexity.

\begin{theorem}
Satisfiability of $\mathrm{GRA}(e,p,s, I,\neg,\exists)$  
can be checked by a finite automaton, i.e., the 
set of satisfiable terms is in \textsc{REG}.
\end{theorem}
\begin{proof}
Consider first terms that are built up starting from $e$ and
using $p,s,I,\neg$ and $\exists$. If neither $\exists$ nor $I$ 
occurs in such a term, it is clearly satisfiable, so we may consider
the case where the term has $\exists$ or $I$ in it. Let $\mathcal{T}$
denote the term under investigation. Suppose $\mathcal{T}$ is of
the form $f_1\dots f_n\, g_1 \dots g_m( e )$ where $f_n\in \{\exists, I\}$
and $g_1,\dots , g_m \in \{p,s,\neg\}$, so the subterm $g_1\dots g_m( e )$ is a
term of $\mathrm{GRA}(e,p,s,\neg)$. Notice that on every model,
the interpretation of $g_1\dots g_m( e )$ is the identity relation if
the number of negations in $g_1\dots g_m$ is even, and otherwise it is
the complement of the identity relation. Notice also that the complement of
the identity relation is the empty binary relation iff the model $\mathfrak{A}$ in
question has a singleton domain, i.e., $|A| = 1$.

First suppose that $f_n = I$. Now, on every model $\mathfrak{M}$,
we have $(I g_1\dots g_m( e ))^\mathfrak{M}= \bot_1^M$ if $g_1\dots g_m$ has an odd number of
negations, and otherwise $(I g_1\dots g_m( e ))^\mathfrak{M}= \top_1^M$.
Therefore, if $g_1\dots g_n$ has an odd number of negations,
the term $$\mathcal{T} = f_1\dots f_n g_1 \dots g_m(e)$$ is 
satisfiable iff the number of negations in $f_1\dots f_{n-1}$ is odd.
If $g_1\dots g_n$ has an even number of negations,
the term $\mathcal{T}$ is 
satisfiable iff the number of negations in $f_1\dots f_{n-1}$ is even.

Suppose then that $f_n = \exists$, and consider
satisfiability on models $\mathfrak{M}$ with a singleton domain.
Restricting to such models, the criteria for
satisfiability are obtained in the same way as in the above case for $f_n = I$, as
now $(\exists\, g_1\dots g_m( e ))^\mathfrak{M}= \bot_1^M$ if $g_1\dots g_m$ has an odd number of
negations, and otherwise $(\exists\, g_1\dots g_m( e ))^\mathfrak{M}= \top_1^M$.
Thus we consider satisfiability over models with at least two elements.
On every such model $\mathfrak{N}$, we
always have $(\exists\, g_1\dots g_m( e ))^\mathfrak{N}=  \top_1^N$.
Therefore the term $\mathcal{T} = f_1\dots f_n g_1 \dots g_m(e)$ with $f_n = \exists$ is
satisfiable on a model with at least two elements iff
the number of negations in $f_1\dots f_{n-1}$ is even.
Bringing the cases for singleton and non-singleton domains together, we
see that with $f_n = \exists$, the term $\mathcal{T}$ is satisfiable if (1) 
the number of negations in $f_1\dots f_{n-1}$ is even, or (2) both $g_1\dots g_m$
and $f_1\dots f_{n-1}$ have an odd number of negations.

Based on the above, it is easy to see how to construct a finite
automaton for checking satisfiability of terms built from $e$ by 
applying $p,s,I,\neg,\exists$. To conclude the proof, we will show that all terms of type $h_1\dots h_n( R )$ of $\mathrm{GRA}(e,p,s,I,\neg,\exists)$, where $R$ is $k$-ary relation symbol, 
are satisifiable.

Now, define two $\{R\}$-models $\mathfrak{A}$ and $\mathfrak{B}$,
both having the same singleton domain $A=\{a\}$ but 
with $R^{\mathfrak{A}} = \top_k^A$ and $R^{\mathfrak{B}} = \bot_k^A$.
In any model $\mathfrak{M}$ with a singleton domain, every $n$-ary 
term $\mathcal{T}$ of $\mathrm{GRA}$ can receive only two 
interpretations, $\bot_n^M$ or $\top_n^M$.
Using this observation and
induction over the structure of terms $\mathcal{T}$
formed from $R$ with $p,s,I,\neg,\exists$, we
can show that $\mathcal{T}^\mathfrak{A} = \bot_n^A$ 
iff $(\neg\mathcal{T})^\mathfrak{B} = \bot_n^A$.
This implies, in particular, that
every such term $\mathcal{T}$ is
satisfiable, since if $\mathcal{T}^\mathfrak{A} = \bot_n^A$, 
then $(\neg\mathcal{T})^\mathfrak{B} = \bot_n^A$, 
and thus $\mathcal{T}^\mathfrak{B} = \top_n^A$.
\end{proof}

We will then study $\mathrm{GRA}$ without $I$
and show that it is
\textsc{NP}-complete.
We begin by identifying a new decidable
fragment $\mathcal{F}$ of $\mathrm{FO}$. The newly identified logic $\mathcal{F}$
turns out to be an interesting, low complexity fragment of $\mathrm{FO}$, as we will
prove it \textsc{NP}-complete. The 
fragment $\mathcal{F}$ is defined as follows.\smallskip
%
%
%
\begin{enumerate}
%
    \item 
   $R(x_1,\dots ,x_n)\in \mathcal{F}$ and $x_1 = x_2 \in \mathcal{F}$ for all relation symbols $R$
    and all variables $x_1,\dots , x_n$.
    \item If \scalebox{0.90}[1.1]{$\varphi, \psi \in \mathcal{F}$}
              and \scalebox{0.9}[1]{$Free(\varphi) \cap Free(\psi) = \varnothing$},
              then \scalebox{0.9}[1]{$(\varphi \wedge \psi) \in \mathcal{F}$}.
    \item If $\varphi \in \mathcal{F}$, then
    $\neg\varphi\in\mathcal{F}$.
    \item
    If $\varphi \in \mathcal{F}$, then $\exists x \varphi \in \mathcal{F}$ for any $x$.\smallskip
\end{enumerate}

\medskip

We  then give a series of lemmas, 
ultimately showing \textsc{NP}-completeness of the system 
$\mathrm{GRA}(e,p, s,\neg,J, \exists)$.
We first 
show that the satisfiability problem of $\mathcal{F}$ is complete for \textsc{NP}.
The upper bound is based on a reduction to the satisfiability
problem of the set of \emph{relational Herbrand sentences}.
These are FO-sentences of the form $$Q_1x_1\dots Q_nx_n \bigwedge_i \eta_i$$
where $Q_i\in \{ \exists, \forall \}$ are quantifiers and each $\eta_i$ is a first-order literal.
We note the fact that checking satisfiability of equality-free relational
Herbrand sentences is known to be  \textsc{PTime}-complete,
see Theorem 8.2.6 in \cite{borger97}. However, there seems to be no explicit
proof of the \textsc{PTime}-completeness of case with equality in the literature, so we provide it in the appendix
(see Lemma \ref{HERBRANDEQUALITY} and the auxiliary Lemma \ref{HERBRANDONEELEMENT}
before that).

\begin{lemma}\label{flemma}
The satisfiability problem of $\mathcal{F}$ is \textsc{NP}-complete.
\end{lemma}
\begin{proof}
The lower bound follows by the fact that if $\varphi$ is a
propositional logic formula, we obtain an
equisatisfiable formula of $\mathcal{F}$ by replacing each
proposition symbol $p_i$ by the sentence $\forall x P_i(x)$.

We thus consider the upper bound.
Let $\chi \in \mathcal{F}$ be a formula. Start by transforming $\chi$ into negation normal form, thus obtaining a formula $\chi'$. Now note that in $\mathcal{F}$, the
formula $\forall x (\varphi \lor \psi)$ is equivalent to either $(\varphi \lor \psi)$, $(\forall x \varphi \lor \psi)$ or $(\varphi \lor \forall x \psi)$ since $Free(\varphi) \cap Free(\psi) = \varnothing$. Similarly, $\exists x (\varphi \wedge \psi)$ is equivalent to $(\varphi \wedge \psi)$, $(\exists x \varphi \wedge \psi)$ or $(\varphi \wedge \exists x \psi)$. Thus we can push all quantifiers past all
connectives in the formula $\chi'$ in polynomial time, getting a 
formula $\chi''$.

Consider now the following elementary trick.
Let $C$ be the set of all conjunctions obtained from $\chi''$ as
follows: begin from the syntax tree of $\chi''$ and keep eliminating
disjunctions $\vee$, always keeping one of the two disjuncts.
Now $\chi''$ is satisfiable iff some $\beta\in C$ is satisfiable.
Starting from $\chi''$, we nondeterministically guess some $\beta\in C$
(without constructing~$C$).

Now, $\beta$ is a 
conjunction of formulas $Q_1x_1\dots Q_kx_k\, \eta$
where $Q_i\in\{\forall,\exists\}$ for each $i$
and $\eta$ is a literal.
Putting $\beta$ in prenex normal form, we get a 
\emph{relational
Herbrand sentence}.
Lemma \ref{HERBRANDEQUALITY} in the appendix
proves that satisfiability of relational Herbrand sentences is
complete for \textsc{PTime}. 
\end{proof}

\begin{lemma}
$\mathrm{GRA}(e,p,s,\neg, J,\exists)$-terms translate 
to equisatisfiable formulas of $\mathcal{F}$ in polynomial time.
\end{lemma}
\begin{proof}
5
    We use induction on the structure of
    terms $\mathcal{T}$ of $\mathrm{GRA}(e,p,s,\neg, J,\exists)$.
    We translate every $k$-ary term to a formula $\chi(v_1,\dots ,v_k)$, so
    the free variables are precisely $v_1,\dots , v_k$. For the base case we note that $e$ is
    equivalent to $v_1 = v_2$ and $R$ to $R(v_1,\dots , v_k)$.
    Suppose then that $\mathcal{T}$ is
    equivalent to $\varphi(v_1,\dots ,v_k)$. 
    Then $\neg \mathcal{T}$ is equivalent to $\neg \varphi(v_1,\dots,v_k)$
    and $\exists \mathcal{T}$ to $\exists v_k \varphi(v_1,\dots,v_k)$. 
    We translate $s\mathcal{T}$
    to  the variant of $\varphi(v_1,\dots ,v_{k})$ that swaps $v_{k-1}$
    and $v_{k}$ and $p\mathcal{T}$ 
    to $\varphi(v_2,\dots ,v_k,v_1)$.
    Finally, suppose that $\mathcal{T}$ translates to
    $\varphi(v_1,\dots ,v_k)$ and $\mathcal{P}$ to $\psi(v_1,\dots ,v_{\ell})$.
    Now $J(\mathcal{T},\mathcal{P})$ is
    translated to $$\varphi(v_1,\dots ,v_k)
    \wedge \psi(v_{k+1},\dots ,v_{k+\ell}).$$
This concludes the proof.
\end{proof}

\begin{lemma}
The satisfiability problem of $\mathrm{GRA}(\neg,J,\exists)$ is \textsc{NP}-hard.
\end{lemma}
\begin{proof}
We give a simple reduction from SAT (which is similar to the reduction in Lemma \ref{flemma}). Let $\varphi$ be a formula of
propositional logic. Let $\{p_1,\dots ,p_n\}$ be the set of proposition symbols in $\varphi$,
and let $\{P_1,\dots ,P_n\}$ be a set of unary
relation symbols. Let $\varphi^*$ be the formula
obtained from $\varphi$ by replacing each symbol $p_i$ with $\forall x P_i(x)$. It is
easy to see that $\varphi$ and $\varphi^*$ are
equisatisfiable. Finally, since $\forall x P_i(x)$ is equivalent to $\neg \exists \neg P_i$, we see that the 
sentence $\varphi^*$ can be expressed in $\mathrm{GRA}(\neg, J, \exists)$.
\end{proof}

Thereby we have now finally proved the
following result for $\mathrm{GRA}$
without $I$.

\begin{corollary}
The satisfiability
problem of $\mathrm{GRA}(e,p,s,\neg, J ,\exists)$ is \textsc{NP}-complete.
\end{corollary}

\section{Undecidable fragments of GRA}

In this section we identify undecidable 
subsystems of $\mathrm{GRA}$. We 
begin with $\mathrm{GRA}$ without $e$.

\begin{proposition}\label{nouproposition}
The satisfiability problem of $\mathrm{GRA}(p, s, I,
\neg, J, \exists)$ is $\Pi_1^0$-complete.
\end{proposition}
\begin{proof}
Analyzing the proof that $\mathrm{GRA}$ 
corresponds to $\mathrm{FO}$, we easily observe that $\mathrm{GRA}(p, s, I,\neg, J, \exists)$ is
equiexpressive with the equality-free fragment of $\mathrm{FO}$ (which is
well known to be $\Pi_1^0$-complete).
\end{proof}

We then consider $\mathrm{GRA}$ without $s$.
To this end, we will
show that the satisfiability problem of $\mathrm{GRA}(p,I,\neg,J,\exists)$ is $\Pi_1^0$-complete.
%
Let us begin by recalling the \textbf{tiling problem} for $\mathbb{N}\times\mathbb{N}$.
A \textbf{tile} is a function $t:\{R,L,T,B\}\rightarrow C$ where $C$ is a countably infinite set of colors.
We let $t_X$ denote $t(X)$. Intuitively, $t_R,t_L,t_T$ and $t_B$
correspond to the colors of the right,
left, top and bottom edges of a tile.
Now, let $\mathbb{T}$ be a finite set of tiles. A \textbf{$\mathbb{T}$-tiling} of $\mathbb{N}\times \mathbb{N}$ is a function $f:\mathbb{N}\times \mathbb{N} \rightarrow \mathbb{T}$ such that for all $i,j\in\mathbb{N}$, we
have $t_R = t'_L$ when $f(i,j) = t$ and $f(i+1,j) = t'$, and
similarly, $t_T = t'_B$ when $f(i,j) = t$ and $f(i,j+1) = t'$.
Intuitively, the right color of each tile equals the left color of its right
neighbour, and analogously for top and bottom colors.
%
%
%
%
%
%
The tiling problem for the grid $\mathbb{N}\times \mathbb{N}$ asks,
with the input of a finite set $\mathbb{T}$ of tiles, if
there exists a $\mathbb{T}$-tiling of $\mathbb{N}\times \mathbb{N}$. It is well known that this problem is $\Pi_1^0$-complete. We will show that the satisfiability problem for $\mathrm{GRA}(p,I,\neg,J,\exists)$ is
undecidable by reducing the tiling problem to it.

Define the \textbf{standard grid} $\mathfrak{G}_{\mathbb{N}}
:= (\mathbb{N} \times \mathbb{N}, R, U)$ where we have
$R = \{((i,j),(i+1,j)) \mid i,j \in \mathbb{N}\}$ and
$U = \{((i,j),(i,j+1)) \mid i,j \in \mathbb{N}\}$.
%
%
%
%
%
%
If $\mathfrak{G}$ is a structure of the
vocabulary $\{R,U\}$ with binary relation symbols $R$ and $U$, 
then $\mathfrak{G}$ is
\textbf{grid-like} if there
is a homomorphism
$\tau: \mathfrak{G}_{\mathbb{N}} \rightarrow \mathfrak{G}$.
Consider then the extended vocabulary $\{R,U,L,D\}$ where $L$ and $D$ are
binary. Define

\medskip

\noindent 
$\varphi_{inverses}    : =  \forall x \forall y (R(x,y) \leftrightarrow L(y,x))$\\ \smallskip
                                       $\text{ }$\hspace{2.3cm}$\wedge\ \forall x\forall y(U(x,y) \leftrightarrow D(y,x))$\\  
\smallskip
\noindent
$\varphi_{successor}     : = \forall x (\exists y R(x,y) \wedge \exists y U(x,y))$\\
\smallskip
\scalebox{0.95}[1]{$\varphi_{cycle}    : =\ \ \forall x\forall y\forall z \forall u [(L(y,x)
\wedge U(x,z) \wedge R(z,u))
      \rightarrow D(u,y)].$}

\medskip

\smallskip

\noindent
Then define $\Gamma := \varphi_{inverses}
\wedge \varphi_{successor}
\wedge\varphi_{cycle}$. 
The intended model of $\Gamma$ is the
standard grid $\mathfrak{G}_{\mathbb{N}}$
extended with two binary relations, $L$
pointing left and $D$ pointing down.

\begin{lemma}\label{undeckey}
Let $\mathfrak{G}$ be a structure of the vocabulary $\{R,U,L,D\}$.
Suppose $\mathfrak{G}$ satisfies $\Gamma$. Then there exists a homomorphism from $\mathfrak{G}_{\mathbb{N}}$ to
$\mathfrak{G} \upharpoonright \{R,U\}$, i.e,. to the
restriction of $\mathfrak{G}$ to
the vocabulary $\{R,U\}$.
\end{lemma}
\begin{proof}
As $\mathfrak{G}$ satisfies $\varphi_{inverses}$ and $\varphi_{cycle}$, it is easy to see that $\mathfrak{G}$ satisfies the sentence
\begin{align*}\varphi_{grid-like}\ := & \\
\forall x\forall y\forall z \forall u & [(R(x,y) \wedge U(x,z) \wedge R(z,u))\rightarrow U(y,u)].
\end{align*}
\noindent
Using this sentence and $\varphi_{successor}$, it is
easy to inductively
construct a homomorphism
from $\mathfrak{G}_{\mathbb{N}}$ to $\mathfrak{G}\upharpoonright \{R,U\}$.
\end{proof}

The sentence $\varphi_{grid-like}$ used above reveals the 
\emph{key trick} in our argument towards proving
undecidability of $\mathrm{GRA}(p,I,\neg,J,\exists)$. 
The sentence $\varphi_{grid-like}$ would be the natural 
choice for our argument 
rather than $\varphi_{cycle}$. Indeed, we
could replace $\Gamma = \varphi_{inverses}
\wedge\varphi_{successor}\wedge\varphi_{cycle}$ in the 
statement of Lemma \ref{undeckey} by 
$\varphi_{successor}
\wedge \varphi_{grid-like}$, as the proof of the
lemma shows. But translating $\varphi_{grid-like}$ to $\mathrm{GRA}(p,I,\neg,J,\exists)$ 
becomes challenging due to the arrangement of the variables and the lack of $s$ in the algebra.
We solve this issue by using  $\varphi_{cycle}$ 
instead of $\varphi_{grid-like}$. By extending the vocabulary, we
can formulate $\varphi_{cycle}$ so that the variables in it
occur in a cyclic order. The proof of Theorem \ref{undecnoswap} below demonstrates 
that---indeed---by using this cyclicity, we can
express $\varphi_{cycle}$ in $\mathrm{GRA}(p,I,\neg,J,\exists)$ even
though it lacks the swap operator~$s$.

Fix a set of tiles $\mathbb{T}$. We simulate the
tiles $t\in \mathbb{T}$ by unary relation symbols $P_t$.
Let $\varphi_{\mathbb{T}}$
be the conjunction of the following
four sentences (the
second one could be dropped):

\medskip

\smallskip

$\forall x \bigvee_{t\in \mathbb{T}} P_t(x),$

\smallskip

$\bigwedge_{t\neq t'} \forall x  \neg (P_t(x) \wedge P_{t'}(x)),$

\smallskip

$\bigwedge_{t_{R} \neq t'_{L}} \forall x \forall y
\neg (P_t(x) \wedge R(x,y) \wedge P_{t'}(y)),$

\smallskip

%
$\bigwedge_{t_{T} \neq t'_{B}} \forall x
\forall y \neg (P_t(x) \wedge U(x,y) \wedge P_{t'}(y)).$

\medskip

\smallskip

\noindent
Now $\varphi_{\mathbb{T}}$ expresses that $\mathbb{N}\times \mathbb{N}$ is $\mathbb{T}$-tilable:

\begin{lemma}
$\mathbb{N}\times \mathbb{N}$ is $\mathbb{T}$-tilable
iff\, $\varphi_{\mathbb{T}}\wedge \Gamma$ is satisfiable.
\end{lemma}
\begin{proof}
Suppose there is a model $\mathfrak{G}$ so that $\mathfrak{G}\models \varphi_{\mathbb{T}}\wedge\Gamma$.
Therefore, by Lemma \ref{undeckey}, there exists a homomorphism $\tau: \mathfrak{G}_{\mathbb{N}} \rightarrow \mathfrak{G} \upharpoonright \{R,U\}$. Define a tiling $T$ of $\mathbb{N}\times \mathbb{N}$ by setting $T((i,j)) = t$ if $\tau((i,j)) \in P_t$. Since $\mathfrak{G}\models \varphi_{\mathbb{T}}$ and $\tau$ is homomorphism, the tiling is well-defined and correct.

Now suppose that there is a tiling $T$ of $\mathbb{N}\times \mathbb{N}$ using $\mathbb{T}$. Thus we can expand $\mathfrak{G}_\mathbb{N} = (\mathbb{N}\times \mathbb{N}, R, U)$ to $\mathfrak{G}'_\mathbb{N} = (\mathbb{N}\times \mathbb{N}, R, U, L, D, (P_t)_{t\in \mathbb{T}})$ in the obvious way. Clearly $\mathfrak{G}'_\mathbb{N} \models \varphi_{\mathbb{T}}\wedge\Gamma$.
\end{proof}

We are now ready to prove the following theorem.

\begin{theorem}\label{undecnoswap}
The satisfiability problem of $\mathrm{GRA}(p,I,\neg,J,\exists)$ is $\Pi_1^0$-complete.
\end{theorem}
\begin{proof}
The upper bound follows by $\mathrm{GRA}(p,I,\neg,J,\exists)$
being contained in $\mathrm{FO}$. For the lower
bound, we will establish that $\varphi_{\mathbb{T}}$ and
each sentence in $\Gamma$ can be expressed in $\mathrm{GRA}(p,I,\neg,J,\exists)$.
%

Now, note that $\varphi_{inverses}$ is
equivalent to the conjunction of the four sentences

\smallskip

       $\forall x \forall y (R(x,y)\rightarrow L(y,x)) \wedge \forall x \forall y (L(y,x)\rightarrow R(x,y)),$  

       $\forall x \forall y (U(x,y)\rightarrow D(y,x)) \wedge \forall x \forall y (D(y,x)\rightarrow U(x,y)).$
       
\smallskip

\noindent
Let us show how to express $\forall x \forall y (R(x,y)\rightarrow L(y,x))$ in $\mathrm{GRA}(p,I,\neg,J,\exists)$;
the other conjuncts are treated similarly.
Consider the formula
$R(x,y)\rightarrow L(y,x)$.
To express this, consider first the formula
$\psi := R(x,y)\rightarrow L(z,u)$
which can be expressed by the term $\mathcal{T} = \neg J (R,\neg L)$.
Now, to make $\psi$ equivalent to $R(x,y)\rightarrow L(y,x)$, we
could first write $y = z \wedge x = u \wedge \psi$ and then existentially
quantify $z$ and $u$ away. On the algebraic side, an essentially
corresponding trick is done by
transitioning from $\mathcal{T}$ first to $I p \, (\mathcal{T})$
and then to $I p I p \, (\mathcal{T})$ and 
finally reordering this by $p$, i.e., going to $ p I p I p \, (\mathcal{T})$.
This term is
equivalent to $R(x,y) \rightarrow L(y,x)$.
%
%
%
Therefore the sentence
$\forall x \forall y (R(x,y)\rightarrow L(y,x))$
is equivalent to $\forall \forall
p I p I p\, \mathcal{T}$ where $\forall = \neg \exists \neg$.


Consider then the formula $\varphi_{cycle} =\ \
\forall x\forall y\forall z \forall u [(L(y,x)
\wedge U(x,z) \wedge R(z,u)) \rightarrow D(u,y)]$.
In the quantifier-free part, the variables occur in a
cyclic fashion, but with repetitions. We first 
translate the repetition-free variant $(L(v_1,v_2)
\wedge U(v_3,v_4) \wedge R(v_5,v_6)) \rightarrow D(v_7,v_8)$ by
using  $\neg$ and $J$, letting $\mathcal{T}$ be the resulting term.
Now we would need to modify $\mathcal{T}$ so that the
repetitions are taken into account. 
To introduce one repetition, first use $p$ on $\mathcal{T}$ repeatedly to
bring the involved coordinates to the
right end of tuples, and then use $I$.
Here $p$ suffices (and $s$ is not needed) because $\varphi_{cycle}$
was designed so that the repeated variable occurrences are
cyclically adjacent to 
each other in the variable ordering.
Thus it is now easy to see that we can form a term $\mathcal{T}'$
equivalent to $(L(v_1,v_2)
\wedge U(v_2,v_3) \wedge R(v_3,v_4)) \rightarrow D(v_4,v_1)$, and $\mathcal{T}'$ can
easily be modified to a term for $\varphi_{cycle}$.

From subformulas of $\varphi_{\mathbb{T}}$, consider the
formula $\neg (P_t(x) \wedge R(x,y) \wedge P_{t'}(y))$.
Here $\psi(x,y) :=  R(x,y) \wedge P_{t'}(y)$ is
equivalent to $\mathcal{T} := I J(R, P_{t'})$ 
and $P_t(x) \wedge \psi(x,y)$
thus to $p I J(p\mathcal{T},P_t)$.
The rest of $\varphi_{\mathbb{T}}$ and the other
remaining formulas are now easy to translate.
\end{proof}

\begin{corollary}
The satisfiability
problem of $\mathrm{GRA}(e,p,I,\neg,J,\exists)$ is $\Pi_1^0$-complete.
\end{corollary}

\section{Conclusion}

The principal aim of the article has been to introduce
our program that facilitates a systematic 
study of logics via algebras based on 
finite signatures.

After presenting $\mathrm{GRA}$, we proved it equivalent to $\mathrm{FO}$.
We also provided algebraic
characterizations for $\mathrm{FO}^2$, $\mathrm{GF}$
and fluted logic and introduced a
general notion of a \emph{relation operator}.
We then provided a comprehensive classification of the 
decidability of subsystems of $\mathrm{GRA}$. 
Out of the cases obtained by removing one operator, only 
the case for $\mathrm{GRA}$ without $p$ was left open.
In each solved case we also identified the related complexity.
Our work can be continued into many directions; the key is to 
identify relevant collections of \emph{relation operators}---as
defined above---and provide classifications for the
thereby generated systems. This work can naturally involve 
systems that capture $\mathrm{FO}$, but also stronger, weaker
and orthogonal ones. In addition to decidability, complexity 
and expressive power, also completeness of 
equational theories (including the one for $\mathrm{GRA}$) is an
interesting research direction. Furthermore, model checking
of different particular systems and
model comparison games for general as well as 
particular sets of relation operators are
relevant topics.

\medskip

\medskip

\medskip

\noindent
\textbf{Acknowledgements.}
The authors were supported by the Academy of Finland project \emph{Theory of
computational logics}, grant numbers 324435, 328987, 352419, 352420, 353027.



\bibliographystyle{plain}
\bibliography{alg}




\medskip


%

%

\medskip

\section{Appendix: An algebra for fluted logic}\label{flutedappendix}

\begin{definition}\label{fluteddefinition}
Here we provide the definition of fluted logic ($\mathrm{FL}$) as
given in \cite{flutedlidiatendera}.
Fix the infinite
sequence $\overline{v}_\omega = (v_1,v_2,\dots)$ of
variables. For every $k\in \mathbb{N}$, we
define sets $\mathrm{FL}^{k}$ as follows.
\begin{enumerate}
    \item Let $R$ be an $n$-ary relation symbol and  
    consider the subsequence $$(v_{k-n+1}, \dots , v_k)$$ 
    of $\overline{v}_\omega$ containing precisely $n$ variables.
    Then $R(v_{k-n+1}, \dots , v_k) \in \mathrm{FL}^k$.
    \item For every $\varphi, \psi \in \mathrm{FL}^k$, we
    have that $\neg \varphi, (\varphi \land \psi) \in \mathrm{FL}^k$.
    \item If $\varphi \in \mathrm{FL}^{k+1}$, then $\exists v_{k+1} \varphi \in \mathrm{FL}^k$.
\end{enumerate}
Finally, we define the set of
fluted formulas to be $\mathrm{FL} := \bigcup_k \mathrm{FL}^k$.
\end{definition}

\begin{proposition}\label{flutedcharacterizationappendix}
$\mathrm{FL}$ and $\mathrm{GRA}(\neg,\Dot{\cap},\exists)$ are equiexpressive.
\end{proposition}
\begin{proof}
We first translate formulas to algebraic terms.
Formulas of the form $$R(v_{k-n+1},\dots ,v_k)$$ are
translated to $R$.
Note that when if $R$ has arity $0$,
then $R(v_{k-0+1},v_k)$ of course
denotes the formula $R$
(which translates to the term $R$).
Suppose then that $\neg \varphi, (\varphi
\wedge \psi) \in \mathrm{FL}^k$ and
that we have translated $\varphi$ to $\mathcal{T}$ and $\psi$ to $\mathcal{S}$.
We translate $\neg \varphi$ to $\neg \mathcal{T}$.
Now, observe that if $\alpha \in \mathrm{FL}^k$, then the free
variables of $\alpha$ form some
suffix of the sequence $(v_1,\dots , v_k)$.
Thus we can translate $(\varphi \land \psi)$ to $(\mathcal{T}\, \Dot{\cap}\,
\mathcal{S})$. Finally, if $\exists v_{k+1} \varphi
\in \mathrm{FL}^k$ and $\varphi$ translates to $\mathcal{T}$, then we can
translate $\exists v_{k+1} \varphi$ to $\exists\mathcal{T}$.

We then translate algebraic terms into fluted logic. 
An easy way to describe 
the translation is by
giving a family of translations $f_{v_m,\dots ,v_k}$ 
where $(v_m,\dots ,v_k)$ is a 
suffix of $(v_1,\dots , v_k)$. (It is also
possible that $f_{v_m,\dots ,v_k} = f_{v_{k+1},v_k}$ which
happens precisely when translating a term of arity zero.)
The translations are as follows.
\begin{enumerate}
    \item $f_{v_m,\dots,v_k}(R) := R(v_m,\dots,v_k)$
    for $ar(R)= {k-m+1}$. (When $\mathit{ar}(R) = 0$,
    then $R$ translates to $R$.)
    \item $f_{v_m,\dots,v_k}(\neg \mathcal{T})
    := \neg f_{v_m,\dots,v_k}(\mathcal{T})$ for $ar(\mathcal{T})= {k-m+1}$.
    \item $f_{v_m,\dots,v_k}(\mathcal{T}\, \Dot{\cap}\, \mathcal{S}) := f_{v_n,\dots,v_{k}}(\mathcal{T})\, 
    \wedge\, f_{v_{\ell},\dots,v_{k}}(\mathcal{S})$\ 
    for \ $ar(\mathcal{T}\, \Dot{\cap}\, \mathcal{S})= {k-m+1};$
    $ar(\mathcal{T})= {k-n+1}$;
    and 
    $ar(\mathcal{S})= {k-\ell+1}$.
    \item $f_{v_{m},\dots,v_k}(\exists \mathcal{T}) = \exists v_{k+1} f_{v_m,\dots,v_{k+1}}(\mathcal{T})$
    for $ar(\exists\mathcal{T})= {k-m+1}$.
\end{enumerate}
This concludes the proof.
\end{proof}

\medskip

\section{Appendix: An algebra for unary negation fragment}\label{unfoappendix}

We will start by formally defining the syntax of unary negation fragment.

\begin{definition}
    We define unary negation fragment ($\mathrm{UNFO}$) as the set of formulas generated by the following grammar
    \[\varphi ::= x = y \mid R(x_1,\dots,x_k) \mid \varphi \land \varphi \mid \varphi \lor \varphi \mid \exists x \varphi \mid \neg \varphi(x),\]
    where $R$ is a relation symbol and the formula $\varphi(x)$ has at most one free variable.
\end{definition}

Using the technique of Theorem \ref{equiexpressivitytheorem} one can easily show the following.

\begin{proposition}\label{unfocharacterizationappendix}
    $\mathrm{UNFO}$ and $\mathrm{GRA}(e,p,s,I,\neg_1,J,\overline{J},\exists)$ are equiexpressive.
\end{proposition}
\begin{proof}
    Clearly terms of $\mathrm{GRA}(e,p,s,I,\neg_1,J,\overline{J},\exists)$ can be translated to $\mathrm{UNFO}$. For the converse direction one can imitate the proof of Theorem \ref{equiexpressivitytheorem}. Using $e,p,s$ and $I$ we can again translate arbitrary atomic formulas into equivalent terms. Conjunctions and disjunctions can be translated using $J$ and $\overline{J}$ respectively together with $p,s$ and $I$. Existential quantifiers can be translated using $p$ and $\exists$. Finally, the unary negation can be translated by just using $\neg_1$.
\end{proof}

\medskip

\section{Appendix: 
Proof of Theorem \ref{theorem:unfo-fluted-complexity}}\label{appendix:unfo-fluted-complexity}

Here we prove Theorem \ref{theorem:unfo-fluted-complexity}. That is, we show that the satisfiability problem of $\mathrm{GRA}(e,s,\backslash,\Dot{\cap},\exists)$ is \textsc{ExpTime}-complete by designing an alternating polynomial space Turing machine which solves it (the lower bound being clear). Before starting, we first go through some useful notation that we are going to employ. We will use $\exists^n$ to denote a sequence of length $n$ that consists only of the symbol $\exists$, and we will use $\forall^n$ to denote $\neg \exists^n \neg$. Furthermore we will use $\cup$ and $\Dot{\cup}$ to denote the dual operators of $\cap$ and $\Dot{\cap}$, which are clearly definable using $\cap, \Dot{\cap}$ and $\neg$. Finally, given two $k$-ary terms $\mathcal{T}$ and $\mathcal{P}$, we will use $\mathcal{T} \models \mathcal{P}$ to denote that $\mathcal{T}$ entails $\mathcal{P}$, i.e., for every model $\mathfrak{A}$ and $\overline{a} \in A^k$ we have that if $\overline{a} \in \mathcal{T}^\mathfrak{A}$ then $\overline{a} \in \mathcal{P}^\mathfrak{A}$.

We begin by introducing a suitable version of Scott normal form for our algebra.

\begin{definition}
    We say that a $0$-ary term $\mathcal{T} \in \mathrm{GRA}(e,s,I,\backslash,\neg,\Dot{\cap},\exists)$ is in \textbf{normal form}, if it has the following form
    \[\bigcap_{i=1}^{n_\exists^1} \exists \kappa_i \cap \bigcap_{j=1}^{n_\forall^1} \forall \lambda_j \cap \bigcap_{i = 1}^{n_\exists} \forall^{n_i} (\neg \alpha_i \cup \exists \mathcal{P}_i) \cap \bigcap_{j = 1}^{n_\forall} \forall^{n_j} (\neg \beta_j \Dot{\cup} (\gamma_j \cup \forall(\neg \delta_j \cup \mathcal{S}_j))),\]
    where $\alpha_j$ are relation symbols, $\beta_j$ are either $Ie$ or a term of $\mathrm{GRA}(s,I)$, $\kappa_i, \lambda_j, \gamma_j, \delta_j$ are terms of $\mathrm{GRA}(e,s,I)$ and $\mathcal{P}$ and $\mathcal{S}$ are terms of $\mathrm{GRA}(e,s,\backslash,\neg,\Dot{\cap})$. The terms $\forall \lambda_j$ are called \textbf{unary universal requirements} while the terms $\forall^{n_j} (\neg \beta_j \cup (\gamma_j \cup \forall(\neg \delta_j \cup \mathcal{S}_j)))$ are called \textbf{polyadic universal requirements}.
\end{definition}

We will require the following easy technical lemma, which is essentially a simplified version of Lemma \ref{termguardlemma}.

\begin{lemma}\label{lemma:baby-term-guard}
    Let $\mathcal{T} \in \mathrm{GRA}(e,s,I,\backslash,\Dot{\cap})$ be a $k$-ary term. Then there exists a $k$-ary term $\alpha \in \mathrm{GRA}(e,s,I)$ such that $\mathcal{T}$ is equivalent with $\alpha \cap \mathcal{T}$.
\end{lemma}
\begin{proof}
    We prove the existence of $\alpha$ using induction. If $\mathcal{T}$ is $R$, for some relation symbol $R$, or $e$, then we can set $\alpha := \mathcal{T}$. Suppose then that the claim holds for $\mathcal{T}$ and $\mathcal{P}$, i.e., there exists $\alpha, \beta \in \mathrm{GRA}(e,s)$ such that $\mathcal{T}$ and $\mathcal{P}$ are equivalent with $\alpha \cap \mathcal{T}$ and $\beta \cap \mathcal{P}$ respectively. For $s\mathcal{T}$ we can choose the term $s\alpha$. For $I\mathcal{T}$ we can choose the term $I\alpha$. For $(\mathcal{T} \backslash \mathcal{P})$ we can take $\alpha$ (even if $ar(\mathcal{T}) \neq ar(\mathcal{P})$). Finally, if $ar(\mathcal{T}) \geq ar(\mathcal{P})$, then for $(\mathcal{T} \Dot{\cap} \mathcal{P})$ we can choose $\alpha$, and otherwise we can choose $\beta$.
\end{proof}

The following lemma is fairly standard, but proving it in our setting seems to require a quite involved argument.

\begin{lemma}
    There is a polynomial time nondeterministic procedure, taking as its input a $0$-ary term $\mathcal{T}_1 \in \mathrm{GRA}(e,s,\backslash,\Dot{\cap},\exists)$ and producing a $0$-ary term $\mathcal{T}_2 \in \mathrm{GRA}(e,s,I,\backslash,\neg,\Dot{\cap},\exists)$ in normal form (possibly over an extended signature) such that the following two conditions hold.
    \begin{enumerate}
        \item If $\mathcal{T}_1^\mathfrak{A}$ is non-empty, for some model $\mathfrak{A}$, then there exists a run of the procedure which produces a term $\mathcal{T}_2$ so that $\mathcal{T}_2^\mathfrak{A}$ is non-empty for some expansion of $\mathfrak{A}$.
        \item If the procedure has a run which produces $\mathcal{T}_2$, then for every model $\mathfrak{A}$ for which $\mathcal{T}_2^\mathfrak{A}$ is non-empty, we have that also $\mathcal{T}_1^\mathfrak{A}$ is non-empty.
    \end{enumerate}
\end{lemma}
\begin{proof}
    Choose an innermost subterm of $\mathcal{T}$ which is of the form $\exists \mathcal{P}$, where $\mathcal{P} \in \mathrm{GRA}(e,s,\backslash,\Dot{\cap})$. If $\mathcal{P}$ is unary, then we will guess a truth value, and replace $\exists \mathcal{P}$ with either $\bot$ or $\top$ according to this guess. If the resulting term is $\mathcal{T}'$, then $\mathcal{T}$ is equi-satisfiable with either $\mathcal{T}' \cap \exists \mathcal{P}$ or $\mathcal{T}' \cap \neg \exists \mathcal{P}$.
    
    Consider then the case where the arity of $\mathcal{P}$ is strictly greater than one. Let $\alpha \in \mathrm{GRA}(e,s)$ be the term guaranteed by Lemma \ref{lemma:baby-term-guard}. If $\alpha$ is equivalent with $e$, then we will replace $\mathcal{P}$ with $I\mathcal{P}$ in $\mathcal{T}$. 
    
    Suppose then that $\alpha$ is not equivalent with $e$, i.e., it is a term of $\mathrm{GRA}(s,I)$. Suppose that $\exists \mathcal{P}$ is an $n$-ary term. We now replace $\exists \mathcal{P}$ with a fresh $n$-ary relation $Q$. If $\mathcal{T}'$ denotes the resulting term, then $\mathcal{T}$ is equi-satisfiable with
    \[\mathcal{T}' \cap \forall^n(\neg Q \cup \exists (\alpha \cap \mathcal{P})) \cap \forall^n (Q \cup \forall (\neg \alpha \Dot{\cup} \neg \mathcal{P})).\]
    Note that $\forall^n (Q \cup \forall (\neg \alpha \Dot{\cup} \neg \mathcal{P}))$ is not yet in normal form. 
    
    Since $\mathcal{T}'$ is a $0$-ary term, there must be a subterm $\exists \mathcal{S}$ of $\mathcal{T}'$ such that $Q$ occurs as a subterm of $\mathcal{S}$. Now we will repeat the rewriting process that we did on $\exists \mathcal{P}$ on all of the (proper) subterms of $\mathcal{S}$ which are of the form $\exists \mathcal{R}$. If $\mathcal{S}'$ denotes the resulting term, then we will use $\mathcal{T}''$ to denote the term obtained from $\mathcal{T}'$ by replacing $\mathcal{S}$ with $\mathcal{S}'$. Note that $\mathcal{S}' \in \mathrm{GRA}(e,s,I,\backslash,\Dot{\cap})$. Thus, using Lemma \ref{lemma:baby-term-guard} again, we can deduce that it has a guard $\beta \in \mathrm{GRA}(e,s,I)$ such that $\mathcal{S}'$ is equivalent with $\beta \cap \mathcal{S}'$. 
    
    We have now two cases based on whether or not $\beta$ is equivalent with $e$. If $\beta$ is not equivalent with $e$, and hence is a member of $\mathrm{GRA}(s,I)$, then we first replace $\exists \mathcal{S}'$ with $\exists (\beta \cap \mathcal{S}')$, after which for every fresh relation symbol $Q$, that was introduced during the previous rewriting step, we replace the corresponding universal requirements
    \[\forall^n (Q \cup \forall (\neg \delta \Dot{\cup} \neg \mathcal{R}))\]
    with the following universal requirements
    \[\forall^n (\neg \beta \Dot{\cup} (Q \cup \forall (\neg \delta \Dot{\cup} \neg \mathcal{R}))),\]
    which are in normal form.
    
    Consider then the case where $\beta$ is equivalent with $e$. First, we replace $\exists \mathcal{S}'$ with the term $I\mathcal{S}'$. Then, for every fresh relation symbol $Q$ that was introduced during the previous rewriting step, we replace the corresponding universal requirements
    \[\forall^n (Q \cup \forall (\neg \delta \Dot{\cup} \neg \mathcal{R}))\]
    with the following universal requirements
    \[\forall^{n-1} (IQ \cup \forall (e \Dot{\cup} P)) \cap \forall^n (\neg P \cup \forall (\neg \delta \Dot{\cup} \neg \mathcal{R})),\]
    where $P$ is a fresh binary relation symbol. We note that the term $\forall^{n-1} (IQ \cup \forall (e \Dot{\cup} P))$ is still not in normal form. However, if $IQ$ is a unary term, then we can rewrite it as a term in normal form by employing $Ie$ as a guard.
    
    Suppose then that $IQ$ is at least a binary relation. Using again the fact $\mathcal{T}''$ is $0$-ary, $I\mathcal{S}'$ must occur as a subterm in a term of the form $\exists \mathcal{S}''$. Thus we can repeat the rewriting process that we did on $\exists \mathcal{S}$ on $\exists \mathcal{S}''$ until we either reach a term $\beta$ in $\mathrm{GRA}(s,I)$ that we can use as a guard or we reach a situation where $Ie$ can be used as a guard.
    
    Having done the previous rewriting, we reach a term of the form $\mathcal{T}_1 \cap \mathcal{T}_2$, where $\mathcal{T}_2$ is in normal form. If $\mathcal{T}_1$ is in normal form, then we are done. Otherwise we repeat the whole rewriting process on $\mathcal{T}_1$.
\end{proof}

Before we present the promised alternating polynomial space Turing machine for solving the satisfiability problem of $\mathrm{GRA}(e,s,\backslash,\Dot{\cap},\exists)$, we will need to introduce several technical definitions. Let $\tau$ be a relational vocabulary. A $k$-ary $(s,I)$-\textbf{atom over} $\tau$ is an $\ell$-ary term in $\mathrm{GRA}(s,I)$, where $\ell \leq k$. A $k$-ary $(s,I)$-\textbf{literal over} $\tau$ is a term in $\mathrm{GRA}(s,I,\neg)$ which is either of the form $\alpha$ or $\neg \alpha$, where $\alpha$ is a $k$-ary $(s,I)$-atom over $\tau$. For $k\geq 2$, we define $k$-ary $(s,I)$-\textbf{table} $\tau$ as a maximal set $\rho$ of $k$-ary $(s,I)$-literals of arity at least two over $\tau$ for which $\Dot{\bigcap}_{\alpha \in \rho} \alpha$ is satisfiable. We will identify tables $\rho$ with terms $\Dot{\cap}_{\alpha \in \rho} \alpha$. $1$-\textbf{type} is a maximally consistent set of unary $(s,I)$-literals over $\tau$. We will identify $1$-types $\pi$ with terms $\bigcap_{\alpha \in \pi} \alpha$. For $k\geq 2$, a $k$-\textbf{type} over $\tau$ is a triple $(\pi_1,\pi_2,\rho)$, where $\pi_1,\pi_2$ are $1$-types over $\tau$ and $\rho$ is a $k$-ary $(s,I)$-table over $\tau$, such that $\pi_1 \Dot{\cap} s(\pi_2 \Dot{\cap} \rho)$ is satisfiable. We will identify types $(\pi_1,\pi_2,\rho)$ with terms $\pi \Dot{\cap} s(\pi_2 \Dot{\cap} \rho)$. For notational ease, we will sometimes use $(\pi_1,\pi_2,\rho)$ to also denote the $1$-type $\pi_2$.

Given a $k$-type $(\pi_1,\pi_2,\rho)$ and a $(k+1)$-type $(\pi_1',\pi_2',\rho')$, we say that the latter is an \textbf{extension} of the former, if $\pi_1' = \pi_2$. Consider a $k$-type $(\pi_1,\pi_2,\rho)$ over $\tau$. A \textbf{reduct} of $(\pi_1,\pi_2,\rho)$ is an $(\ell + p)$-type $(\pi_1',\pi_2',\rho')$, such that one the following conditions holds.
\begin{enumerate}
    \item $\rho'$ consists of the set of all $p$-ary $(s,I)$-literals $\alpha$ for which there exists an unary term in $\pi_2$ which is equivalent with $I^p\alpha$, and $\pi_1' = \pi_2' = \pi_2$.
    \item $\rho'$ consists of the set of all $(\ell+p)$-ary $(s,I)$-literals $\alpha$ for which there exists an $\ell$-ary term in $\rho$ which is equivalent with $I^p \alpha$, and if $p = 0$, then $\pi_1' = \pi_1$ and $\pi_2' = \pi_2'$, and if $p > 0$, then $\pi_1' = \pi_2' = \pi_2$.
    \item $\rho'$ consists of the set of all $(\ell+p)$-ary $(s,I)$-literals $\alpha$ for which there exists an $\ell$-ary term in $\rho$ which is equivalent with $sI^p\alpha$, and if $p = 0$, then $\pi_1' = \pi_2$ and $\pi_2' = \pi_1$, and if $p > 0$, then $\pi_1' = \pi_2' = \pi_1$.
\end{enumerate}
Both in the first and in the second case we say that $(\pi_1',\pi_2',\rho')$ is a \textbf{simple reduct} of $(\pi_1,\pi_2,\rho)$.

Observe that so far types do not impose any equality constraints. To take these constraints into account, we will introduce two useful pieces of notation. First, if $\xi$ is a $k$-type and $\mathcal{P} \in \mathrm{GRA}(e,s,I,\backslash,\Dot{\cap})$ is a $k$-ary term, then we write $\xi \models^{=} \mathcal{P}$, if $\neg (e \Dot{\cap} (\xi \cap \mathcal{P}))$ is not satisfiable. Similarly, we write $\xi \models^{\neq} \mathcal{P}$, if $\neg (\neg e \Dot{\cap} (\xi \cap \mathcal{P}))$ is not satisfiable.

Consider a unary universal requirement $\forall \lambda_j$. We say that a type $\xi = (\pi_1,\pi_2,\rho)$ \textbf{violates} this universal requirement, if $\pi_2 \cup \lambda_j$ is not satisfiable. Consider then a polyadic universal requirement
\[\forall^{n_j} (\neg \beta_j \Dot{\cup} (\gamma_j \cup \forall(\neg \delta_j \cup \mathcal{S}_j)))\]
and let $\xi$ be a type and let $\xi'$ be its extension. We say that the pair $(\xi,\xi')$ \textbf{violates} this universal requirement, if one of the following conditions holds.
\begin{enumerate}
    \item $ar(\xi) = ar(\beta_j)$, $\xi \models \beta_j \Dot{\cap} \neg \gamma_j$ and $\xi' \models \delta_j$, but $\xi' \not\models^{\neq} \mathcal{P}_i$.
    \item There exists a $ar(\beta_j)$-type $\xi''$, which is a reduct of $\xi'$, and a $(ar(\beta_j) + 1)$-ary simple reduct $\xi'''$ of $\xi'$ such that $\xi'' \models \beta_j \Dot{\cap} \neg \lambda_j$ and $\xi''' \models \delta_j$, but $\xi''' \not\models^{=} \mathcal{S}_j$.
\end{enumerate}

The following estimate will play a crucial role when we will analyze the space requirements of our algorithm.

\begin{lemma}\label{lemma:no-of-literals}
    Let $\tau$ be a relational vocabulary of size $m$. Now the number of non-equivalent $(s,I)$-atoms over $\tau$ is at most $2mK$, where $K = \max \{ar(R) \mid R \in \tau\}$. Furthermore, the number of non-equivalent types over $\tau$ is bounded above by $2^{m(2K+2)}$.
\end{lemma}
\begin{proof}
    We start by noting that for every term $\mathcal{T}$ we have that $ss\mathcal{T}$ is equivalent with $\mathcal{T}$ and $Is\mathcal{T}$ is equivalent with $I\mathcal{T}$. In particular, every $(s,I)$-atom over $\tau$ is (up to equivalence) either of the form $I^n R$ or $sI^n R$, for some $k$-ary $R\in \tau$, where $k\geq n$. For a fixed $k$-ary relation $R$ the number of such terms is bounded above by $2k$ and hence the number of $(s,I)$-atoms is bounded above by $2mK$. The furthermore part follows from the fact that every $1$-type and every table is uniquely determined via the subset of $(s,I)$-atoms that it contains.
\end{proof}

Let $\mathfrak{A}$ be a $\tau$-model of vocabulary $\tau$ and let $(a_1,\dots,a_k) \in A^k$. We say that $(a_1,\dots,a_k)$ \textbf{realizes} a $k$-type $(\pi_1,\pi_2,\rho)$, if $a_{k-1} \in \pi_1^\mathfrak{A}, a_k \in \pi_2^\mathfrak{A}$ and $(a_1,\dots,a_k) \in \rho^\mathfrak{A}$. Notice that each $k$-tuple realizes a unique type. The type realized by a tuple $(a_1,\dots,a_k)$ in a model $\mathfrak{A}$ is denoted by $\mathrm{tp}_\mathfrak{A}(a_1,\dots,a_k)$.

Now we are finally in a position where we can define the promised alternating polynomial space Turing machine for solving the satisfiability problem of the system $\mathrm{GRA}(e,s,\backslash,\Dot{\cap},\exists)$. As its input the algorithm receives a $0$-ary term $\mathcal{T}$ in normal form; let $\tau$ denote the vocabulary of $\mathcal{T}$. We set $N_\tau := 2^{|\varphi|(2|\varphi|+2)}$. To simplify the presentation, we will assume that $n_\exists^1 = 1$.

\begin{enumerate}
    \item [] \textbf{Existentially guess} a $1$-type $\xi$ over $\tau$.
    \item [] \textbf{If} $\xi \not \models \kappa_1$ \textbf{then reject}
    \item [] \textbf{If} $\xi$ violates any unary universal requirements of $\mathcal{T}$ \textbf{then reject}
    \item [] \textbf{For} $1\leq c\leq N_\tau + 1$ \textbf{do}
    \begin{enumerate}
        \item [] \textbf{Universally choose} $1\leq i\leq n_\exists$ and a reduct $\xi'$ of $\xi$ so that $\alpha_i \in \xi'$
        \item [] Set $\xi'''$ to be the $(ar(\mathcal{P}_i))$-ary simple reduct of $\xi'$
        \item [] \textbf{If} $\xi''' \models^{=} \mathcal{P}_i$ \textbf{then skip}
        \item [] \textbf{Existentially guess} an extension $\xi''$ of $\xi'$
        \item [] \textbf{If} $\xi'' \not\models^{\neq} \mathcal{P}_i$ \textbf{then reject}
        \item [] \textbf{If} $(\xi',\xi'')$ violates any universal requirements of $\mathcal{T}$ \textbf{then reject}
        \item [] Set $\xi := \xi''$
    \end{enumerate}
    \item [] \textbf{Accept}
\end{enumerate}

It is straightforward to verify that the above algorithm is complete, i.e., if the input term is satisfiable, then the algorithm will accept. The following lemma shows that the algorithm is also sound.

\begin{lemma}
    If the above algorithm accepts a given $0$-ary term $\mathcal{T}$ in normal form, then the term is satisfiable.
\end{lemma}
\begin{proof}
    Suppose that the existential player $\exists$ has a winning strategy $\sigma$ in the game played over the configuration graph of the above algorithm when it receives the $0$-ary $\mathcal{T}$ in normal form as its input. Without loss of generality we can clearly assume that the extensions that $\sigma$ recommends $\exists$ to select depend only on the current type $\xi$ and not on the value of $c$.
    
    Now let $M$ denote the set of all $(s,I)$-types that $\sigma$ instructs $\exists$ to choose in different positions of the game. Since the for-loop is executed for $N_\tau + 1$ steps, where $N_\tau$ is an upper bound on the number of types over $\tau$, and the extensions that $\sigma$ recommends to $\exists$ do not depend on the value of $c$, every type in $M$ occurs as the current type $\xi$ in the for-loop at least once.
    
    Using $\sigma$ and $M$, we will construct a tree-like model for $\mathcal{T}$ in a step-by-step manner. More formally, we will construct a sequence of models
    \[\mathfrak{A}_0 \leq \mathfrak{A}_1 \leq \dots\]
    where $\mathfrak{A} \leq \mathfrak{B}$ means that $\mathfrak{A}$ is a submodel of $\mathfrak{B}$, in such a way that their union $\mathfrak{A}$ has the property that $\mathcal{T}^\mathfrak{A}$ is non-empty. During the construction we will maintain the invariant that all the types of \emph{live} tuples of these models, by which we mean either singletons or tuples that belong to the interpretation of some term from $\mathrm{GRA}(s,I)$, are from $M$. 
    
    We start the construction of this sequence as follows. Let $\xi$ denote the first $1$-type that $\sigma$ instructs $\exists$ to choose. As the first model $\mathfrak{A}_0$ we will take a model whose domain consists of a single element $a$ which satisfies the property that $\mathrm{tp}_{\mathfrak{A}_0}[a] = \xi$. Since $\sigma$ is a winning strategy for $\exists$, $\xi$ does not violate any unary universal requirements of $\mathcal{T}$.
    
    Suppose then that we have already constructed $\mathfrak{A}_n$ and we need to construct $\mathfrak{A}_{n+1}$. To achieve this, we need to provide witnesses for $n$-fresh tuples of $\mathfrak{A}_n$. Let $(a_1,\dots,a_k)$ be an arbitrary $n$-fresh tuple of $\mathfrak{A}_n$ and let $\xi$ denote the type that it realizes in $\mathfrak{A}_n$ and let $1\leq i\leq n_\exists$. We say that $(a_1,\dots,a_k)$ is an $i$-defect, if there exists a reduct $\xi'$ of $\xi$ so that $\alpha \in \xi'$, but $\xi''' \not\models^{=} \mathcal{P}_i$, where $\xi'''$ is the $ar(\mathcal{P}_i)$-ary simple reduct of $\xi'$. Our goal is to provide witnesses for $i$-defects.
    
    Suppose that $(a_1,\dots,a_k)$ is an $i$-defect. If $\xi$ denotes its type, then we let $\xi'$ denote a reduct of $\xi$ for which $\alpha_i \in \xi'$. Now there exists $\ell \leq k$ and $p \geq 0$ such that either $(a_\ell,\dots,a_{k-1},a_k,\dots,a_k)$ or $(a_\ell,\dots,a_k,a_{k-1},\dots,a_{k-1})$, where the suffix $(a_k,\dots,a_k)$ contains $p + 1$ elements, realizes $\xi'$. Suppose that the first tuple realizes $\xi'$ (the other case being entirely analogous). Let $\xi''$ be the type that $\sigma$ instructs player $\exists$ to choose. We now extend the model $\mathfrak{A}_n$ by adding a fresh element $b$ and by specifying that
    \[\mathrm{tp}_{\mathfrak{A}_{n+1}}(a_\ell,\dots,a_{k-1},a_k,\dots,a_k,b) = \xi''.\]
    Notice that since $\xi''$ is an extension of $\xi'$, the element $a_k$ has the same $1$-type both in $\xi'$ and in $\xi''$. Since $\sigma$ is a winning strategy, we have that $(a_\ell,\dots,a_{k-1},a_k,\dots,a_k,b) \in \mathcal{P}_i^{\mathfrak{A}}$. Furthermore, since $\sigma$ is a winning strategy, this assignment does not violate any of the universal requirements of $\mathcal{T}$.
    
    Having provided witnesses for all the $i$-defects, we note that the resulting model $\mathfrak{A}_{n+1}^*$ is still incomplete, since there might be single elements for which we have not yet specified their $1$-types (recall that by definition the types that we are using intuitively only specify the $1$-types of the last two elements). To each such element we will simply assign the $1$-type that we assigned to the single element in $\mathfrak{A}_0$. Having done this, there might still be tuples $(a_1,\dots,a_k)$ and relation symbols $R\in \tau$ such that we have not yet specified whether or not $(a_1,\dots, a_k)$ belongs to the interpretation of $R$. We complete the structure $\mathfrak{A}_{n+1}^*$ in a minimal way by specifying that in every such case the tuple does not belong to the interpretation of the corresponding relation symbol. The resulting model will be selected as the model $\mathfrak{A}_{n+1}$.
\end{proof}

To conclude, we note that since the binary encoding of $N_\tau$ is of size at most polynomial with respect to $|\mathcal{T}|$ and all the types over $\tau$ are also of size at most polynomial with respect to $|\mathcal{T}|$, our algorithm clearly uses only polynomial space. This completes the proof of Theorem \ref{theorem:unfo-fluted-complexity}.

\section{Appendix: Relational Herbrand formulas}

\begin{lemma}\label{HERBRANDONEELEMENT}
    Checking whether a relational Herbrand sentence has a model of domain size one can be done in polynomial time.
\end{lemma}
\begin{proof}
    Over models of size one, every relational Herbrand sentence is equivalent to a
sentence of the form $\forall x \bigwedge_i \eta_i$ where each $\eta_i$ is either an atomic formula or a negation of such. If the sentence contains a formula $\neg x = x$, then it is unsatisfiable. Otherwise such a sentence is satisfiable iff it does not contain two complementary formulas $\eta_i$ and $\neg \eta_i$, which can be easily verified in polynomial time.
\end{proof}

\begin{lemma}\label{HERBRANDEQUALITY}
    The satisfiability problem for relational 
    Herbrand sentences is \textsc{PTime}-complete.
\end{lemma}
\begin{proof}
The satisfiability problem for equality-free relational Herbrand formulas is \textsc{PTime}-complete by
Theorem 8.2.6 of \cite{borger97}, so the lower bound follows from that result. Thus we only need to 
discuss the upper bound here. We will show this by a reduction to equality-free Herbrand sentences, essentially by
showing how equalities can be eliminated.
Let $\varphi$ be a relational Herbrand sentence. If 
one of the conjuncts in $\varphi$ is a literal of type $\neg x=x$, then $\varphi$ is
unsatisfiable. And if one of the literals is $x=x$, we can remove it and thereby obtain an
equisatisfiable sentence. (In the extreme special case where all the conjuncts of $\varphi$
are equalities of type $x=x$, we cannot remove all of them, but in this case $\varphi$ is 
clearly satisfiable.) Therefore we may assume that each equality
and negated equality in $\varphi$ has two \emph{different} variables in it.

We also assume that every variable in $\varphi$ is quantified precisely once.
Moreover, for notational convenience, we assume that if $\varphi$
contains an atomic formula $x_i = x_j$ (or a negation of such), then the
quantifier binding $x_j$ appears in the scope of the quantifier binding $x_i$.
If $x_i$ is a variable appearing in $\varphi$, then we say
that it is existentially (respectively, universally) quantified if the quantifier binding it is existential (respectively, universal).

We then show how to eliminate and deal with equalities and negated equalities from $\varphi$ by considering a number of cases.

    Suppose first that $\varphi$ contains a formula $\neg x_i = x_j$, so that one of the following cases holds.
    \begin{enumerate}
        \item $x_i$ is existentially and $x_j$ universally quantified.
        \item $x_i$ and $x_j$ are both universally quantified.
    \end{enumerate}
    Then clearly $\varphi$ is not satisfiable, and we can reject it.
Suppose then that $\varphi$ contains a formula $x_i = x_j$ so that one of
the two conditions below holds.
    \begin{enumerate}
  \item $x_i$ is existentially and $x_j$ universally quantified.
        \item $x_i$ and $x_j$ are both universally quantified.
    \end{enumerate}
    Then $\varphi$ is clearly satisfiable iff it has a model of size one, which can be checked in polynomial time by lemma \ref{HERBRANDONEELEMENT}.
    
    In the remaining cases, we can eliminate equalities
and negated equalities from $\varphi$ as follows.
First suppose that $\varphi$ contains at least one instance of a formula $\neg x_i = x_j$ so that either
    \begin{enumerate}
        \item $x_i$ is universally and $x_j$ existentially quantified, or
        \item $x_i$ and $x_j$ are both existentially quantified.
    \end{enumerate}
    Then we introduce a fresh binary relation symbol $E$ and replace all such formulas $\neg x_i = x_j$ with the formula $\neg E(x_i,x_j)$. If $\varphi^*$ denotes the resulting formula, then clearly $\varphi$ is equisatisfiable with $\forall x E(x,x) \land \varphi^*$, which can be easily transformed to a relational Herbrand sentence.
    
    Suppose then that $\varphi$ contains an instance of a formula $x_i = x_j$ so that 
    \begin{enumerate}
        \item $x_i$ is universally and $x_j$ existentially quantified, or
        \item $x_i$ and $x_j$ are both existentially quantified.
    \end{enumerate}
    In both cases we can remove the existential quantifier which is binding the variable $x_j$ from the
sentence $\varphi$ and replace $x_j$ with $x_i$ in every atomic
formula that appears in $\varphi$.

    Thus we have essentially shown, for all types of 
equalities and negated equalities, 
how they can be eliminated from $\varphi$ so that an 
equisatisfiable formula is obtained.
Since the resulting sentence belongs to the Herbrand fragment without equality, its
satisfiability can be checked in polynomial time.
\end{proof}

\end{document}